\documentclass[12pt]{amsart}
\usepackage{verbatim}
\usepackage{amssymb}
\usepackage{upref}
\usepackage[cmtip,all]{xy}
\usepackage{hyperref}
\usepackage[usenames,dvipsnames]{color}
\usepackage[normalem]{ulem}

\allowdisplaybreaks

\newtheorem{thm}{Theorem}[section]
\newtheorem{lem}[thm]{Lemma}
\newtheorem{cor}[thm]{Corollary}
\newtheorem{prop}[thm]{Proposition}

\theoremstyle{definition} 
\newtheorem{defn}[thm]{Definition}
\newtheorem{q}[thm]{Question}
\newtheorem{ex}[thm]{Example}

\newtheorem{rem}[thm]{Remark}

\numberwithin{equation}{section}

\newcommand{\secref}[1]{Section~\textup{\ref{#1}}}

\newcommand{\thmref}[1]{Theorem~\textup{\ref{#1}}}
\newcommand{\corref}[1]{Corollary~\textup{\ref{#1}}}
\newcommand{\lemref}[1]{Lemma~\textup{\ref{#1}}}
\newcommand{\propref}[1]{Proposition~\textup{\ref{#1}}}

\newcommand{\defnref}[1]{Definition~\textup{\ref{#1}}}

\newcommand{\KK}{\mathcal K}

\newcommand{\MM}{\mathcal M}
\newcommand{\NN}{\mathcal N}

\newcommand{\C}{\mathbb C}

\newcommand{\variso}{\overset{\simeq}{\longrightarrow}}

\renewcommand{\bar}{\overline}
\newcommand{\what}{\widehat}

\newcommand{\wilde}{\widetilde}
\newcommand{\inv}{^{-1}}

\renewcommand{\epsilon}{\varepsilon}

\renewcommand{\:}{\colon}
\renewcommand{\subset}{\subseteq}

\newcommand{\xt}{\otimes}
\newcommand{\xm}{\otimes_{\max}}

\renewcommand{\)}{\textup)}
\newcommand{\rt}{\textup{rt}}
\newcommand{\lt}{\textup{lt}}
\newcommand{\id}{\text{\textup{id}}}
\newcommand{\triv}{\delta_{\textup{triv}}}

\newcommand{\rify}{{}^R}
\newcommand{\sify}{{}^S}
\newcommand{\deltag}{\delta_G}
\newcommand{\deltagr}{\delta_G^R}
\newcommand{\direct}[1]{\operatorname{R}(#1)}
\newcommand{\perturb}[1]{\wilde{\operatorname R}(#1)}

\DeclareMathOperator{\ad}{Ad}
\DeclareMathOperator*{\spn}{span}
\DeclareMathOperator*{\clspn}{\overline{\spn}}

\DeclareMathOperator{\cpc}{CPC}
\DeclareMathOperator{\cpa}{CPA}
\DeclareMathOperator{\relcom}{C}
\DeclareMathOperator{\cpar}{CPAR}
\DeclareMathOperator{\c-r}{C-R}

\newcommand{\midtext}[1]{\quad\text{#1}\quad}
\newcommand{\righttext}[1]{\quad\text{#1 }}

\newcommand{\eg}{\emph{e.g.}}
\newcommand{\ie}{\emph{i.e.}}

\newcommand{\cgalg}{$C_0(G)$-decorated algebra}
\newcommand{\csgalg}{$C^*(G)$-decorated algebra}
\newcommand{\kalg}{$\KK$-dec\-or\-ated algebra}
\newcommand{\dalg}{$D$-decorated algebra}
\newcommand{\cgco}{$C_0(G)$-fixed coaction}
\newcommand{\kco}{$\KK$-fixed coaction}
\newcommand{\dco}{$D$-fixed coaction}
\newcommand{\cgcor}{$C_0(G)$-fixed R-coaction}
\newcommand{\kcor}{$\KK$-fixed R-coaction}
\newcommand{\dcor}{$D$-fixed R-coaction}

\newcommand{\coaction}{coaction}

\begin{document}
\title
{R-coactions on $C^*$-algebras}
\author[Kaliszewski]{S. Kaliszewski}
\address{School of Mathematical and Statistical Sciences
\\Arizona State University
\\Tempe, Arizona 85287}
\email{kaliszewski@asu.edu}
\author[Landstad]{Magnus~B. Landstad}
\address{Department of Mathematical Sciences\\
Norwegian University of Science and Technology\\
NO-7491 Trondheim, Norway}
\email{magnus.landstad@ntnu.no}
\author[Quigg]{John Quigg}
\address{School of Mathematical and Statistical Sciences
\\Arizona State University
\\Tempe, Arizona 85287}
\email{quigg@asu.edu}

\date{August 19, 2021}


\begin{abstract}
We give the beginnings of the development of a theory of
what we call ``R-coactions''
of a locally compact group
on a $C^*$-algebra.
These are the coactions taking values in the maximal tensor product,
as originally proposed by Raeburn.
We show that the theory has some gaps as compared to the more familiar theory of standard coactions.
However, we indicate how we needed to develop some of the basic properties of R-coactions as a tool in our program involving the use of coaction functors in the study of the Baum-Connes conjecture.
\end{abstract}
\maketitle

\section{Introduction}\label{intro}

Coactions of locally compact groups on $C^*$-algebras have been around for almost 50 years,
and the conventions have evolved in various ways.
Initially, the theory was developed spatially, using reduced group $C^*$-algebras and reduced crossed products.
In the early 1990's, Raeburn \cite{rae:representation}
revolutionized the theory, using universal properties and full group $C^*$-algebras and crossed products.
There was one aspect of Raeburn's new conventions, however, that did not seem to catch on: using maximal rather than minimal tensor products.
More precisely, it has become customary to refer to
Raeburn's innovation as
\emph{full coactions}, which are homomorphisms
$\delta\:A\to M(A\xt C^*(G))$ satisfying certain conditions,
and to the older style as \emph{reduced coactions},
which are homomorphisms
$\delta\:A\to M(A\xt C^*_r(G))$.
Raeburn himself, on the other hand, proposed
that coactions should map into $M(A\xm C^*(G))$.

In the course of
our program to use coactions as an aid in the study of the Baum-Connes conjecture (see, \eg, \cite{klqfunctor, klqfunctor2, klqtensoraction}),
we found at one point that
we needed to find a coaction functor appropriate for the study of
a certain crossed-product functor introduced by Baum, Guentner, and Willett \cite{bgwexact},
the construction of which uses maximal tensor products.
In \cite{klqtensoraction} we handled the case where the group $G$ is discrete,
mainly through the use of Fell bundles.
In our efforts to generalize this to arbitrary locally compact $G$,
it occurred to us that
a homomorphism
$\psi\:B\rtimes_\alpha G\to (B\xm C)\rtimes_{\alpha\xm\gamma} G$
occurring in \cite{bgwexact}
is somehow similar to a coaction
$\delta\:A\to M(A\otimes C^*(G))$,
but with the minimal tensor product $\otimes$
replaced by the maximal one $\xm$.
This insight lead to our approach in
\cite{tensorD}
where we
need to briefly switch to Raeburn-style coactions.
We call these ``R-coactions'' ---
see \secref{r-co sec} for the definition and the comparison to
the usual ``standard'' coactions.
We chose the ``R'' to stand for ``Raeburn''.
We quote from \cite[page~36]{fulldual}:
``it may prove convenient to return to the maximal $C^*$-tensor product at some other time'' --- apparently that time is now.

Eventually we noticed that our paper \cite{tensorD} was becoming overly long, so we decided to separate out the basic development of R-coactions, leading to the current paper.
We develop a few new aspects of the theory of R-coactions,
particularly the notation of cocycles.

We should emphasize that we were motivated to write this paper
in order to have the basic facts concerning R-coactions available.
This has in some ways guided our development.
However, we feel that R-coactions will be useful elsewhere.

We begin in \secref{prelim} by recalling preliminary facts and definitions concerning coactions of a locally compact group $G$ on a $C^*$-algebra $A$, particularly the basics of crossed-product duality.
To prepare for our development of
cocycles for R-coactions, among other topics,
we review in some detail the corresponding facts concerning standard coactions.
We particularly need quite a bit of preparation involving Fischer's procedure for \emph{maximalization} of a coaction,
and we review this in detail in a separate \secref{fischer construction}.

In \secref{r-co sec} we develop the basic theory of R-coactions.
We were surprised by the number of familiar features of standard coactions that do not seem to carry over well --- or at all, in some cases --- to R-coactions.
For instance, although every R-coaction descends to a standard one, we do not know whether every standard coaction arises in this way,
or whether the R-coaction is unique when it exists.
Also, there appears to be no reasonable way to \emph{normalize} an R-coaction.
Finally, although it is possible (see \cite{rae:representation}) to form the crossed product of an R-coaction, it seems that it must be done on Hilbert space, not more generally in the multiplier algebra of a $C^*$-algebra.
In view of these issues, we needed to be quite careful in our use of R-coactions;
indeed, in \cite{tensorD} we only use R-coactions in a limited way, passing to them briefly but then passing back to standard coactions as soon as possible.
We close \secref{r-co sec} with appropriate notions of equivariant homomorphisms and quotients for R-coactions.

In \secref{rification} we show how to pass naturally from maximal coactions to R-coactions.
This requires going through the entire maximalization process,
additionally we must develop a theory of one-cocycles for R-coactions.
The culmination is \thmref{main}, which takes a maximal coaction $(A,\delta)$ and produces an R-coaction $\delta\rify$ on the same $C^*$-algebra $A$.
There is a subtlety: when the maximal coaction is in fact the dual coaction $\wilde\alpha$ on a full crossed product $B\rtimes_\alpha G$,
there is a much easier way to define an R-coaction, and in \thmref{same} we show that fortunately the two methods give the same result in this case.

We close in \secref{conclusion} with a brief explanation of why we do not try to include a theory of crossed-product duality for R-coactions, although Raeburn did prove a duality theorem in \cite{rae:representation}.

\section{Preliminaries}\label{prelim}

Throughout, $G$ denotes a 
locally compact group, and $A,B,\dots$ denote $C^*$-algebras.
We write $\KK$ for the algebra $\KK(L^2(G))$ of compact operators on $L^2(G)$.
We refer to \cite{enchilada,maximal,koqlandstad} for our conventions regarding actions, coactions, $C^*$-correspondences, and cocycles for coactions.
We
recall some notation for the convenience of the reader.

The left and right regular representations of $G$ are $\lambda$ and $\rho$, respectively,
the representation of $C_0(G)$ on $L^2(G)$ by multiplication operators is $M$,
and the unitary element $w_G\in M(C_0(G)\xt C^*(G))$ is the strictly continuous map $w_G\:G\to M(C^*(G))$ given by the canonical embedding of $G$.
(Note that we must be careful to understand from the context whether $M$ is being used for the representation of $C_0(G)$ or for a multiplier algebra.)
The actions of $G$ on $C_0(G)$ by left and right translations are denoted by $\lt$ and $\rt$, respectively.

Throughout this paper, we work with numerous categories of what could be called ``decorated'' $C^*$-algebras,
\ie, the objects are $C^*$-algebras together with some extra structure,
and the morphisms are homomorphisms preserving the extra structure.
To make it interesting, we have to deal with two sorts of morphisms:
\begin{itemize}
\item
``classical'' morphisms, which are
homomorphisms $\phi\:A\to B$ between $C^*$-algebras,
and

\item
nondegenerate morphisms, which are nondegenerate homomorphisms $\pi\:A\to M(B)$.
\end{itemize}
We actually need the classical morphisms for our main results,
and the nondegenerate homomorphisms appear in secondary (albeit crucial) roles.
We will formalize the various categories using classical morphisms,
and occasionally refer to nondegenerate categories more informally.
In an attempt to reduce confusion,
we will not actually use the term ``morphism'' when the $C^*$-algebras are not carrying any extra structure;
in other words, we will write either 
``$\phi\:A\to B$ is a homomorphism'',
or
``$\phi\:A\to M(B)$ is a nondegenerate homomorphism''
when $\phi$ is not required to have any other property.

To help prepare for our ``R-coactions''
we review here in detail some fundamental aspects of ``ordinary'' coactions.
First, a bit of parallel theory of actions:
A \emph{morphism} $\phi\:(A,\alpha)\to (B,\beta)$ of actions is 
an $\alpha-\beta$ equivariant homomorphism
$\phi\:A\to B$,
giving the \emph{category of actions}.
A \emph{covariant representation} of an action $(A,\alpha)$ in a multiplier algebra $M(B)$
is a pair $(\pi,U)$, where $\pi\:A\to M(B)$ is a nondegenerate homomorphism,
$U\:G\to M(B)$ is a strictly continuous unitary representation,
and $\pi\circ \alpha_s=\ad U_s\circ\pi$ for all $s\in G$.
A \emph{crossed product} of an action $(A,\alpha)$ of $G$ is
a 
universal covariant representation $(i_A,i_G)\:(A,G)\to M(A\rtimes_\alpha G)$,
meaning that for every covariant representation
$(\pi,U)\:(A,G)\to M(B)$
there is a unique nondegenerate homomorphism
$\pi\times U\:A\rtimes_\alpha G\to M(B)$,
called the \emph{integrated form} of $(\pi,U)$,
such that
\[
(\pi\times U)\circ i_A=\pi
\midtext{and}
(\pi\times U)\circ i_G=U.
\]
Any two crossed products of $(A,\alpha)$ are uniquely isomorphic,
and in practice a choice is made, and
the $C^*$-algebra $A\rtimes_\alpha G$ is referred to as the crossed product,
but the pair $(i_A,i_G)$ must be kept in mind.
We use superscripts $i_A^\alpha,i_G^\alpha$ if confusion seems possible.
The pair $(M,\rho)$ is a universal covariant representation of
$(C_0(G),\rt)$,
and 
\[
C_0(G)\rtimes_\rt G=\KK.
\]

If $\phi\:(A,\alpha)\to (B,\beta)$ is a morphism of actions,
then there is a unique homomorphism
$\phi\rtimes G\:A\rtimes_\alpha G\to B\rtimes_\beta G$
such that
\[
(\phi\rtimes G)\bigl(i_A(a)i_G^\alpha(c)\bigr)=i_B\circ\phi(a)i_G^\beta(c)\righttext{for all}a\in A,c\in C^*(G).
\]
In this way, the crossed-product construction is functorial from actions to $C^*$-algebras.
Note that we could not just write, \eg, $(\phi\rtimes G)\circ i_A=i_B\circ\phi$, because $\phi$, and consequently $\phi\rtimes G$, could be degenerate.

There is a parallel ``nondegenerate category of actions'',
and crossed products are functorial between nondegenerate categories as well.
However, as we mentioned already, the classical categories are of primary interest in this paper.

Turning to coactions,
first let $\delta_G\:C^*(G)\to M(C^*(G)\xt C^*(G))$ be the integrated form of the unitary representation given by $\delta_G(s)=s\xt s$ for $s\in G$.
Recall the following notational device:
for any $C^*$-algebras $A,D$ we define the \emph{tilde multiplier algebra} as
\[
\wilde M(A\xt D)=\{m\in M(A\xt D):m(1\xt D)\cup (1\xt D)m\subset A\xt D\}.
\]
Note that $\wilde M(A\xt D)$ is not symmetric in $A$ and $D$.
A \emph{coaction} of $G$ on $A$ is a nondegenerate faithful homomorphism
$\delta\:A\to \wilde M(A\xt C^*(G))$
such that
$(\delta\xt\id)\circ\delta=(\id\xt\deltag)\circ\delta$
and
$\clspn\{\delta(A)(1\xt C^*(G))\}=A\xt C^*(G)$.
A \emph{morphism} $\phi\:(A,\delta)\to (B,\epsilon)$ of coactions is a homomorphism $\phi\:A\to B$ that is
\emph{$\delta-\epsilon$ equivariant} in the sense that the diagram
\[
\xymatrix@C+70pt{
A \ar[r]^-\delta \ar[d]_\phi
&\wilde M(A\xt C^*(G)) \ar[d]^{\bar{\phi\xt\id}}
\\
B \ar[r]_-\epsilon
&\wilde M(B\xt C^*(G))
}
\]
commutes,
where the right-hand vertical arrow $\bar{\phi\xt\id}$ is the unique extension,
whose existence is guaranteed by \cite[Proposition~A.6]{enchilada},
of $\phi$ to
$\wilde M(A\xt C^*(G))$.
In this way we get the \emph{category of coactions}.
A \emph{covariant representation} of a coaction $(A,\delta)$ in $M(B)$
is a pair $(\pi,\mu)$
of nondegenerate homomorphisms
\[
\xymatrix{
A \ar[r]^-\pi
&M(B)
&C_0(G) \ar[l]_-\mu
}
\]
such that
$\ad \mu\xt\id(w_G)\circ (\pi\xt 1)=(\pi\xt\id)\circ\delta$.
The \emph{regular representation} of a coaction $(A,\delta)$ in $M(A\xt\KK)$ is given by
\[
\bigl((\id\xt\lambda)\circ\delta,1\xt M\bigr).
\]

A \emph{crossed product} of a coaction $(A,\delta)$ of $G$ is
a 
universal covariant representation $(j_A,j_G)\:(A,C_0(G))\to M(A\rtimes_\delta G)$,
meaning that for every covariant representation
$(\pi,\mu)\:(A,C_0(G))\to M(B)$
there is a unique nondegenerate homomorphism,
$\pi\times \mu\:A\rtimes_\delta G\to M(B)$,
called the \emph{integrated form} of $(\pi,\mu)$,
such that
\[
(\pi\times \mu)\circ j_A=\pi
\midtext{and}
(\pi\times \mu)\circ j_G=\mu.
\]
Any two crossed products of $(A,\delta)$ are uniquely isomorphic,
and in practice a choice is made, and
the $C^*$-algebra $A\rtimes_\delta G$ is referred to as the crossed product,
but the pair $(j_A,j_G)$ must be kept in mind.
The regular representation of any coaction $(A,\delta)$ is universal,
and 
so whenever it is convenient we are free to
identify
$A\rtimes_\delta G$ with a nondegenerate subalgebra of $M(A\xt\KK)$.
We use superscripts $j_A^\delta,j_G^\delta$ if confusion seems possible.

If $\phi\:(A,\delta)\to (B,\epsilon)$ is a morphism of coactions,
then there is a unique morphism
$\phi\rtimes G\:A\rtimes_\delta G\to B\rtimes_\epsilon G$
such that
\[
(\phi\rtimes G)\bigl(j_A(a)j_G^\delta(f)\bigr)=j_B\circ\phi(a)j_G^\epsilon(f)\righttext{for all}a\in A,f\in C_0(G).
\]
As with actions, the crossed-product construction is functorial from coactions to $C^*$-algebras.

We denote the \emph{trivial coaction} 
$a\mapsto a\xt 1$ on $A$ by $\triv$,
and more generally if $(A,\delta)$ is a coaction
and
$B\subset M(A)$,
then $\delta$ is \emph{trivial on $B$} if
$\delta(b)=b\xt 1$ for all $b\in B$.

For an action $(A,\alpha)$,
the \emph{dual coaction} $\what\alpha$ of $G$ on $A\rtimes_\alpha G$
is the integrated form of the covariant representation
$(i_A\xt 1,(i_G\xt\id)\circ\deltag)$, and
is the unique coaction that is trivial on $i_A(A)$
and such that
the nondegenerate homomorphism $i_G\:C^*(G)\to M(A\rtimes_\alpha G)$ is $\deltag-\what\alpha$ equivariant.

The crossed product $\phi\rtimes G\:A\rtimes_\alpha G\to B\rtimes_\epsilon G$
of a morphism $\phi\:(A,\alpha)\to (B,\beta)$ of actions
is $\what\alpha-\what\beta$ equivariant,
and hence gives a morphism $\phi\rtimes G\:(A\rtimes_\alpha G,\what\alpha)\to (B\rtimes_\beta G,\what\beta)$ of coactions,
and in this way the crossed product is functorial from actions to coactions.

For a coaction $(A,\delta)$,
the \emph{dual action} $\what\delta$ of $G$ on $A\rtimes_\delta G$
is defined so that for each $s\in G$ the automorphism $\what\delta_s$
is the integrated form of the covariant representation
$(j_A,j_G\circ\rt_s)$, and
is the unique action that is trivial on $j_A(A)$
and such that the
nondegenerate homomorphism $j_G\:C_0(G)\to M(A\rtimes_\delta G)$ is $\rt-\what\delta$ equivariant.
The crossed product $\phi\rtimes G\:A\rtimes_\delta G\to B\rtimes_\epsilon G$
of a morphism $\phi\:(A,\delta)\to (B,\epsilon)$ of actions
is $\what\delta-\what\epsilon$ equivariant,
and hence gives a morphism $\phi\rtimes G\:(A\rtimes_\delta G,\what\delta)\to (B\rtimes_\epsilon G,\what\epsilon)$ of actions,
and in this way the crossed product is functorial from coactions to actions.
For the trivial coaction $\triv$ on $\C$,
we have
\[
(\C\rtimes_{\triv} G,\what\triv,j_G)=(C_0(G),\rt,\id).
\]

Note that crossed products for both actions and coactions are also functorial between the nondegenerate categories:
\eg, if $\phi\:A\to M(B)$ is a nondegenerate $\alpha-\beta$ equivariant homomorphism
then the above construction gives a 
nondegenerate $\what\alpha-\what\beta$ equivariant homomorphism $\phi\rtimes G\:A\rtimes_\alpha G\to M(B\rtimes_\beta G)$.

The \emph{canonical surjection}
\[
\Phi=\Phi_A\:A\rtimes_\delta G\rtimes_{\what\delta} G\to A\xt \KK
\]
is the integrated form of the covariant representation
\[
\bigl((\id\xt\lambda)\circ\delta\times (1\xt M),1\xt\rho\bigr).
\]
If $\Phi$ is injective, the coaction $\delta$ is called \emph{maximal}.
The trivial coaction $(\C,\triv)$ is maximal, with
\[
\C\rtimes_{\triv} G\rtimes_{\what{\triv}} G=\KK.
\]

A \emph{maximalization} of a coaction $(A,\delta)$ consists of a maximal coaction $(B,\epsilon)$ and a surjective $\epsilon-\delta$ equivariant homomorphism $\psi\:B\to A$
such that the crossed product
$\psi\rtimes G\:B\rtimes_\epsilon G\to A\rtimes_\delta G$
is an isomorphism.
Sometimes the coaction $(B,\epsilon)$ itself is referred to as a maximalization.

\section{Fischer maximalization process}\label{fischer construction}

Our development of R-coactions will depend heavily upon what we call the \emph{Fischer maximalization process}, which we describe
at the end of this section,
after quite a lot of preparation.

We frequently need to compute with possibly degenerate homomorphisms.
Our main tool for handling this is
``$D$-multipliers'', 
which we review below after introducing ``\dalg s''.

\begin{defn}\label{d alg}
Fix a $C^*$-algebra $D$.
A \emph{\dalg} is a pair $(A,\pi)$, where $A$ is a $C^*$-algebra and $\pi\:D\to M(A)$ is a nondegenerate homomorphism.
If $(A,\pi)$ and $(B,\psi)$ are \dalg s then a homomorphism $\phi\:A\to B$ is \emph{$\pi-\psi$ compatible} if
\[
\phi(\pi(d)a)=\psi(d)\phi(a)\righttext{for all}d\in D,a\in A,
\]
in which case we say that $\phi\:(A,\pi)\to (B,\psi)$ is a \emph{morphism} of \dalg s.
\end{defn}

Note that we could not just write $\phi\circ\pi=\psi$,
since $\phi$ might be degenerate.

\begin{rem}
Warning: we will frequently need the case $D=C_0(G)$,
\ie, \cgalg s.
This must not be confused with the more familiar ``$C_0(G)$-algebras'', where $C_0(G)$ maps into the \emph{central} multipliers.
\end{rem}

\begin{lem}\label{d alg cat}
With the above structure, \dalg s and their morphisms form a category.
\end{lem}

\begin{proof}
The only nonobvious thing is that we can compose morphisms:
let $\phi\:(A,\pi)\to (B,\psi)$ and $\sigma\:(B,\psi)\to (C,\eta)$ be morphisms of \dalg s.
We must show that the composition $\sigma\circ\phi\:A\to C$ is a \dalg\ morphism.
This is completely routine, but we give the computation to indicate the style of how these things go.
For $d\in D$ and $a\in A$,
\begin{align*}
(\sigma\circ\phi)(\pi(d)a)
&=\sigma\bigl(\phi(\pi(d)a)\big)
\\&=\sigma\bigl(\psi(d)\phi(a)\bigr)
\\&=\eta(d)\sigma(\phi(a))\righttext{(because $\psi(d)\in C$)}
\\&=\eta(d)\sigma\circ\phi(a).
\qedhere
\end{align*}
\end{proof}

\begin{rem}
Note that the category of \dalg s is \emph{not} a ``coslice'' (or ``comma'') category (see, for example, \cite[Section~II.6]{maclane}),
\ie, it is not the category of all objects ``under'' $D$ in a category of $C^*$-algebras,
because the morphisms $\pi$ and $\phi$ are not of the same type:
we require the $\pi$ to be nondegenerate, but allow it to map into the multiplier algebra $M(A)$,
whereas we require a morphism $\phi$ to map into $B$, and allow it to be degenerate.
\end{rem}

We now apply some of the results of \cite[Appendix~A]{dkq} to \dalg s.

\begin{defn}\label{d mult}
Let $(A,\pi)$ be a \dalg.
A \emph{$D$-multiplier} of $A$ is an element $m\in M(A)$ such that
\[
\pi(D)m\cup m\pi(D)\subset A,
\]
and $M_D(A)$ denotes the set of all $D$-multipliers of $A$.
The \emph{$D$-strict topology} on $M_D(A)$ is generated by the seminorms
\[
m\mapsto \|\pi(d)m\|\midtext{and}m\mapsto \|m\pi(d)\|\righttext{for}d\in D.
\]
\end{defn}

\begin{ex}
For any $C^*$-algebras $A,D$,
\[
\wilde M(A\xt D)=M_D(A\xt D),
\]
where the nondegenerate homomorphism is $d\mapsto 1\xt d$.
\end{ex}

The following lemma is 
\cite[Lemma~A.4]{dkq}, which in turn is based upon \cite[Proposition~A.5]{enchilada}.

\begin{lem}[{\cite{dkq}}]
With the notation from \defnref{d mult},
\begin{enumerate}
\item
The $D$-strict topology on $M_D(A)$ is stronger than the relative strict topology from $M(A)$.

\item
$M_D(A)$ is 
a $C^*$-subalgebra
of $M(A)$,
and multiplication and involution are separately $D$-strictly continuous.

\item
$M_D(A)$ is the $D$-strict completion of $A$.

\item
$M_D(A)$ is an $M(D)$-subbimodule of $M(A)$.
\end{enumerate}
\end{lem}

The whole point of $D$-multipliers is to extend possibly degenerate homomorphisms.
\cite[Lemma~A.5]{dkq} is a quite general result along these lines;
we only need the following special case.

\begin{lem}[\cite{dkq}]\label{extend}
Every morphism $\phi\:(A,\pi)\to (B,\psi)$ of \dalg s extends uniquely to a
$D$-strictly continuous homomorphism $\bar\phi\:M_D(A)\to M_D(B)$.
Moreover, for $d\in D$ and $m\in M_D(A)$,
\[
\phi(\pi(d)m)=\psi(d)\bar\phi(m)
\midtext{and}
\phi(m\pi(d))=\bar\phi(m)\psi(d).
\]
\end{lem}

\begin{cor}\label{md funct}
The assignments $(A,\pi)\mapsto M_D(A)$ and $\phi\mapsto \bar\phi$ give a functor from \dalg s to $C^*$-algebras.
\end{cor}

\begin{proof}
Given morphisms $\phi\:(A,\pi)\to (B,\psi)$ and $\sigma\:(B,\psi)\to (C,\rho)$ of \dalg s,
the composition
\[
\bar\sigma\circ\bar\phi\:M_D(A)\to M_D(C)
\]
is $D$-strictly continuous, and hence coincides with $\bar{\sigma\circ\phi}$ by the uniqueness clause of
\lemref{extend}.
\end{proof}

One natural source of \dalg s involves 
Exel's $C^*$-blends
(see \cite{exelblend}).
For example, in \lemref{ac x pr comp} we have a $C^*$-blend
\[
(A,C^*(G),i_A,i_G^\alpha,A\rtimes_\alpha G).
\]
Recall that a \emph{$C^*$-blend} is a 5-tuple
$\MM=(A,D,i,\pi,X)$,
where $A$, $D$, and $X$ are $C^*$-algebras
and
$i\:A\to M(X)$ and $\pi\:D\to M(X)$ are homomorphisms
such that
\[
X=\clspn\{i(A)\pi(D)\}.
\]
Taking adjoints, we also have $X=\clspn\{\pi(D)i(A)\}$,
and it follows quickly that $i$ and $\pi$ are nondegenerate.

Exel defines morphisms between $C^*$-blends.
However, for our purposes it will be convenient to slightly embellish
the morphisms of \cite{exelblend}:
if $\NN=(B,E,j.\psi,Y)$ is another $C^*$-blend,
and if
\[
\Phi\:X\to Y,\quad\phi\:A\to B,\midtext{and}\gamma\:D\to E
\]
are homomorphisms,
we say that the triple $(\Phi,\phi,\gamma)$ is a \emph{morphism}
from $\MM$ to $\NN$ if
\[
\Phi\bigl(i(a)\pi(d)\bigr)=j\circ\phi(a)\psi\circ\gamma(d)
\righttext{for all}a\in A,d\in D.
\]

\begin{rem}
Note that the definition of $C^*$-blend as above is symmetric in $A$ and $D$.
Also, in fact $i$ maps $A$ into the $D$-multiplier algebra $M_D(X)$
(and similarly $\pi\:D\to M_A(X)$).
For a morphism, most of the time we will have $D=E$ and $\gamma=\id_D$ (symmetrically, we could have $A=B$ and $\phi=\id_A$),
in which case we write a morphism simply as $(\Phi,\phi)$,
which satisfies
\[
\Phi\bigl(i(a)\pi(d)\big)=j\circ\phi(a)\psi(d)\righttext{for all}a\in A,d\in D.
\]
\end{rem}

The following lemma shows that $\Phi$ is a sort
of module map.

\begin{lem}\label{blend1}
If $(\Phi,\phi,\gamma)\:(A,D,i,\pi,X)\to (C,E,j,\psi,Y)$ is a morphism of $C^*$-blends as above,
then
\[
\Phi\bigl(i(a)x\bigr)=j\circ\phi(a)\Phi(x)
\righttext{for all}a\in A,x\in X.
\]
Similarly, $\Phi(\pi(d)x)=\psi\circ\gamma(d)\Phi(x)$ for $b\in D,x\in X$,
and also similarly for products in the opposite order.

In the case $D=E$ and $\gamma=\id_D$,
$\Phi\:(X,\pi)\to (Y,\psi)$ is a morphism of \dalg s.
\end{lem}

\begin{proof}
By linearity and density we can take $x=i(a')\pi(b)$ with $a'\in A,d\in D$.
Then
\begin{align*}
\Phi\bigl(i(a)x\bigr)
&=\Phi(\bigl(i(a)i(a')\pi(d)\bigr)
\\&=\Phi\bigl(i(aa')\pi(d)\bigr)
\\&=\pi\circ\phi(aa')\tau(d)
\\&=\pi\circ\phi(a)\pi\circ\phi(a')\tau(d)
\\&=\pi\circ\phi(a)\Phi\bigl(i(a')\pi(d)\bigr)
\\&=\pi\circ\phi(a)\Phi(x).
\end{align*}
The other assertions follow by symmetry and taking adjoints,
except that the last part is just a special case.
\end{proof}

\begin{lem}\label{blend2}
Let $\phi\:(A,\pi)\to (B,\gamma)$ be a morphism of \dalg s.
Further let $(\Phi,\phi,\gamma)\:(A,D,i,\pi,X)\to (B,E,j,\psi,Y)$ be a morphism
of $C^*$-blends.
Then
\[
\Phi\:(X,i\circ\pi)\to (Y,\psi\circ\gamma)
\]
is also a morphism of \dalg s.
\end{lem}

\begin{proof}
Let $d\in D$ and $x\in X$.
We must show that
\[
\Phi\big(i\circ\pi(d)x\bigr)=\psi\circ\psi(d)\Phi(x).
\]
By the Cohen-Hewitt factorization theorem we can take $x=i(a')x'$
with $a'\in A,x'\in X$.
Then
\begin{align*}
\Phi\big(i\circ\pi(d)x\bigr)
&=\Phi\big(i\circ\pi(d)i(a')x'\bigr)
\\&=\Phi\bigl(i(\pi(d)a')x'\bigr)
\\&=j\circ\phi(\pi(d)a')\Phi(x')
\righttext{(\lemref{blend1})}
\\&=j\bigl(\phi(\pi(d)a'\bigr)\Phi(x')
\\&=j\bigl(\psi(d)\phi(a')\bigr)\Phi(x')
\\&=j\circ\psi(d)j\circ\phi(a')\Phi(x')
\\&=j\circ\psi(d)\Phi\bigl(i(a')x'\bigr)
\\&=j\circ\psi(d)\Phi(x).
\end{align*}
\end{proof}

\begin{rem}
Note that if $\phi$ is nondegenerate then \lemref{blend1} says that
$\Phi\:(X,i)\to (Y,j\circ\phi)$ is a morphism of \dalg s.
\end{rem}

\lemref{blend1} implies the following two well-known lemmas:

\begin{lem}\label{ac x pr comp}
Let $\phi\:(A,\alpha)\to (B,\beta)$ be a morphism of actions.
Then
\[
(\phi\rtimes G)\bigl(i_A(a)x\bigr)=i_B\circ\phi(a)(\phi\rtimes G)(x)
\righttext{for all}a\in A,x\in A\rtimes_\alpha G.
\]
Also,
$\phi\rtimes G\:(A\rtimes_\alpha G,i_G^\alpha)\to (B\rtimes_\beta G,i_G^\beta)$
is a morphism of \csgalg s.
\end{lem}

\begin{lem}\label{co x pr comp}
Let $\phi\:(A,\delta)\to (B,\epsilon)$ be a morphism of coactions.
Then
\[
(\phi\rtimes G)\bigl(j_A(a)x\bigr)=j_B\circ\phi(a)(\phi\rtimes G)(x)
\righttext{for all}a\in A,x\in A\rtimes_\alpha G.
\]
Also,
$\phi\rtimes G\:(A\rtimes_\delta G,j_G^\delta)\to (B\rtimes_\epsilon G,j_G^\epsilon)$
is a morphism of \cgalg s.
\end{lem}

We
will need to know that 
crossed products are compatible with \dalg\ structure,
and 
we record this in the following lemma,
which is a routine application of \lemref{blend2}.

\begin{lem}\label{x pr d comp}
Let $\phi\:(A,\pi)\to (B,\psi)$ be a morphism of \dalg s.
If $\phi$ is also equivariant for actions $\alpha,\beta$,
then
\[
\phi\rtimes G\:(A\rtimes_\alpha G,i_A\circ\pi)\to (B\rtimes_\epsilon G,i_B\circ\psi)
\]
is a morphism of \csgalg s.

Similarly if $\phi$ is equivariant for coactions instead of actions.
\end{lem}

\subsection*{Equivariant actions}

An \emph{equivariant action} is a triple $(A,\alpha,\mu)$,
where $(A,\alpha)$ is an action of $G$ and
$(A,\mu)$ is a \cgalg\
such that $\mu$ is $\rt-\alpha$ equivariant.

Note that the categories of actions and of \cgalg s combine immediately to form a \emph{category of equivariant actions},
where a morphism $\phi\:(A,\alpha,\mu)\to (B,\beta,\nu)$ is just a morphism
$\phi\:(A,\alpha)\to (B,\beta)$ of actions
that is also $\mu-\nu$ compatible.

We are now ready for the first functor that we want to record with a name.

\begin{lem}\label{cpc functor}
The assignments
$(A,\delta)\mapsto (A\rtimes_\delta G,\what\delta,j_G)$
and
$\phi\mapsto \phi\rtimes G$
give a functor 
$\cpc$ 
from coactions to equivariant actions.
\end{lem}

\begin{proof}
This follows immediately from \lemref{co x pr comp},
since we know already that the crossed product construction is functorial from coactions to actions.
\end{proof}

The name $\cpc$ is an acronym for ``crossed product by coactions'',
and
the functor $\cpc$ is the first step in the Fischer maximalization process.

Coactions will combine with \dalg s in several contexts, so it is efficient to introduce the following abstract concept.

\begin{defn}\label{d co}
A \emph{\dco}
is a triple $(A,\delta,\pi)$, where $(A,\delta)$ is a coaction, $(A,\pi)$ is a \dalg,
and
\[
\delta\circ\pi=\pi\xt 1.
\]
\end{defn}

Note that the categories of coactions and of \dalg s combine immediately to form a \emph{category of \dco s},
where a morphism $\phi\:(A,\delta,\pi)\to (B,\epsilon,\psi)$ is just a morphism
$\phi\:(A,\delta)\to (B,\epsilon)$ of coactions
that is also $\pi-\psi$ compatible.

\begin{rem}
One special case we will need is for $D=C_0(G)$.
Warning:
\cgco s must not be confused with
Nilsen's ``$C_0(X)$-coactions'' \cite[Section~3]{nil:full},
which further require $C_0(X)$ to map into the center $ZM(A)$.
\end{rem}

\subsection*{Coaction cocycles}

\begin{defn}\label{cocycle def}
A \emph{cocycle} for a coaction $(A,\delta)$ is a unitary $U\in M(A\xt C^*(G))$
such that
\begin{enumerate}
\item
$(\id\xt\deltag)(U)=(U\xt 1)(\delta\xt\id)(U)$ and

\item
$\clspn\{\ad U\circ\delta(A)(1\xt C^*(G))\}= A\xt C^*(G)$.
\end{enumerate}
\end{defn}
$U$ is also called a \emph{$\delta$-cocycle}.
If $U$ is a cocycle for $(A,\delta)$, then $\ad U\circ\delta$ is a another coaction,
called the \emph{perturbation} of $\delta$ by $U$,
and also is said to be \emph{exterior equivalent} to $\delta$.
There is a subtlety here:
in \defnref{cocycle def}~(ii) it is enough to require only that $\ad U\circ\delta$ maps $A$ into $\wilde M(A\xt C^*(G))$;
it is a theorem that coaction-nondegeneracy of $\ad U\circ\delta$ follows from that of $\delta$
(\eg, see \cite[Proposition~2.5 and the discussion surrounding it]{graded}).

It is still unknown whether cocycle-nondegeneracy is automatic;
perhaps similarly, it is unknown whether \defnref{cocycle def}~(ii) is redundant.

If $(A,\delta)$ and $(B,\epsilon)$ are coactions,
$U$ is a $\delta$-cocycle,
and $\phi\:A\to M(B)$ is a nondegenerate $\delta-\epsilon$ equivariant homomorphism,
then $(\phi\xt\id)(U)$ is an $\epsilon$-cocycle
and $\phi$ is also $\ad U\circ\delta-\ad(\phi\xt\id)(U)\circ\epsilon$ equivariant.
If $U$ is a $\delta$-cocycle and $W$ is an $\ad U\circ\delta$-cocycle,
then $WU$ is a $\delta$-cocycle, and of course $\ad W\circ \ad U\circ\delta=\ad WU\circ\delta$.
Also, clearly $U^*$ is an $\ad U\circ\delta$-cocycle and $\ad U^*\circ\ad U\circ\delta=\delta$.
It follows from the facts recalled in this paragraph that exterior equivalence is an equivalence relation on coactions.

We need a result going back to Nakagami and Takesaki \cite[Theorem~A.1]{naktak}
(see also \cite[Remark~3.2 (2)]{lprs} and
\cite[Lemma~1.2]{qrtwisted}):
for every nondegenerate homomorphism $\mu\:C_0(G)\to M(A)$
the unitary $W=(\mu\xt\id)(w_G)\in M(A\xt C^*(G))$ satisfies
\[
(\id\xt\deltag)(W)=W_{12}W_{13},
\]
where we use the ``leg'' notation:
\[
W_{12}=W\xt 1\midtext{and} W_{13}=(\id\xt\Sigma)(W\xt 1),
\]
and where 
$\Sigma$ is the ``flip automorphism'' of $C^*(G)\xt C^*(G)$ given on elementary tensors by
$\Sigma(x\xt y)=y\xt x$;
moreover, every such unitary $W$ arises in this way from a unique $\mu$.
We will say that $\mu$ and $W$ are \emph{associated to} each other.

Here is a quite fundamental special case:
$w_G$ is the unitary associated to the identity automorphism of $C_0(G)$,
and 
is a cocycle for the trivial coaction $\triv$ on $C_0(G)$.
Moreover, by commutativity the perturbed coaction $\ad w_G\circ\triv$ is again $\triv$.
We combine this with the established theory of coaction cocycles
in the following proposition.

\begin{prop}\label{fn coc}
Let $(A,\delta)$ be a coaction,
and let 
$W$
be the unitary associated to
a nondegenerate homomorphism
$\mu\:C_0(G)\to M(A)$.
Then $W$
is a
$\delta$-cocycle
if and only if
$(A,\delta,\mu)$ is a \cgco.
\end{prop}

\begin{proof}
One direction follows immediately from results quoted in the preceding discussion:
if $(A,\delta,\mu)$ is a \cgco, then $\mu$ is a nondegenerate $\triv-\delta$ equivariant homomorphism,
so $W$ is a $\delta$-cocycle because $w_G$ is a $\triv$-cocycle.

Conversely, assume that $W$ is a $\delta$-cocycle.
For
$U=(\mu\xt\id)(w_G)$ the left-hand side of \defnref{cocycle def}~(i) 
becomes
\begin{align*}
(\id\xt\deltag)\circ(\mu\xt\id)(w_G)
&=(\mu\xt\id\xt\id)\circ(\id\xt\deltag)(w_G)
\\&=(\mu\xt\id\xt\id)\bigl((w_G)_{12}(w_G)_{13}\bigr)
\\&=(\mu\xt\id)(w_G)_{12}(\mu\xt\id)(w_G)_{13},
\end{align*}
while 
the right-hand side 
becomes
\begin{align*}
(\mu\xt\id)(w_G)_{12}(\delta\circ\mu\xt\id)(w_G),
\end{align*}
so 
\defnref{cocycle def}~(i)
implies
(and in fact
is equivalent,
upon cancelling the first factor,
to)
\[
(\mu\xt\id)(w_G)_{13}=(\delta\circ\mu\xt\id)(w_G).
\]
Since
\[
(\mu\xt\id)(w_G)_{13}=\bigl((\mu\xt 1)\xt\id\bigr)(w_G),
\]
we see that \defnref{cocycle def}~(i), together with the Nakagami-Takesaki characterization,
implies that the two homomorphisms
$\delta\circ\mu$ 
and 
$\mu\xt 1$
of $C_0(G)$ coincide,
and hence $(A,\delta,\mu)$ is a \cgco.
\end{proof}

\begin{rem}
The converse direction of \propref{fn coc},
namely that $\delta\circ\mu=\mu\xt 1$ is necessary for $W$ to be a $\delta$-cocycle,
seems to not be previously recorded in the literature; at least, we could not find it.
\end{rem}

\begin{defn}
In \propref{fn coc}, given a \cgco\ $(A,\delta,\mu)$, we will refer to $W$ as the \emph{associated cocycle}.
\end{defn}
Note that we have already been calling $W$ the unitary associated to $\mu$,
so now $W$ becomes associated to two things: a homomorphism $\mu$, and
a \cgco\ $(A,\delta,\mu)$;
these two usages are consistent.

As we mentioned already,
if $(A,\delta)$ and $(B,\epsilon)$ are coactions,
$U$ is a $\delta$-cocycle,
and $\phi\:A\to M(B)$ is a nondegenerate $\delta-\epsilon$ equivariant homomorphism,
then $(\phi\xt\id)(U)$ is an $\epsilon$-cocycle
and $\phi$ is also $\ad U\circ\delta-\ad(\phi\xt\id)(U)\circ\epsilon$ equivariant.
We need a variant of this fact for the particular type of cocycles of interest to us;
it is contained in \cite[Lemma~3.6]{klqfunctor},
but in the following lemma we state it formally for convenient reference in the current paper.

\begin{lem}[\cite{klqfunctor}]\label{mor coc}
If $\phi\:(A,\delta,\mu)\to (B,\epsilon,\nu)$ is a morphism of \cgco s,
with associated cocycles $W$ and $U$, respectively,
then $\phi$ is also $\ad W\circ\delta-\ad U\circ\epsilon$ equivariant.
\end{lem}

We now have a second functor that we want to officially record with a name.

\begin{lem}\label{cpa functor}
Let $(A,\alpha,\mu)$ be an equivariant action,
and let $W$ be the unitary associated to the homomorphism $i_A\circ\mu$.
Then
$W$
is an $\what\alpha$-cocycle.
Put
\[
\wilde\alpha=\ad W\circ\what\alpha.
\]
Then
$(A\rtimes_\alpha G,\wilde\alpha,\mu\rtimes G)$ is a \kco.

Moreover, if $\phi\:(A,\alpha,\mu)\to (B,\beta,\nu)$ is a morphism of equivariant actions,
then
the homomorphism
$\phi\rtimes G$
gives
a morphism
\[
(A\rtimes_\alpha G,\wilde\alpha,\mu\rtimes G)\to
(B\rtimes_\beta G,\wilde\beta,\nu\rtimes G)
\]
of \kco s.

Finally, the assignments
$(A,\alpha,\mu)\mapsto (A\rtimes_\alpha G,\wilde\alpha,\mu\rtimes G)$
and
$\phi\mapsto \phi\rtimes G$
give
a functor $\cpa$
from equivariant actions to \kco s.
\end{lem}

The name $\cpa$ is an acronym for ``crossed product by action'',
and
the functor $\cpa$ is the second step in the Fischer maximalization process.

\begin{proof}
Since $(A\rtimes_\alpha G,\what\alpha,i_A\circ\mu)$ is a \cgco\ by definition of $\what\alpha$,
it follows from \propref{fn coc} that $W$ is a $\what\alpha$-cocycle.

Recall that we identify $C_0(G)\rtimes_\rt G=\KK$,
with $(i_{C_0(G)},i_G)=(M,\rho)$.
Then the definition of dual coactions implies that $(\KK,\what\rt,M)$ is a \cgco,
so 
the associated unitary $U$
is a $\what\rt$-cocycle.
Moreover,
the identity
\[
w_G(1\xt s)=(\rt_s\xt\id)(w_G)
\]
implies that in fact
$\ad U\circ\what\rt$ 
is the trivial coaction on $\KK$.
The homomorphism $\mu\:C_0(G)\to M(A)$ is $\rt-\alpha$ equivariant,
so the crossed-product homomorphism
\[
\mu\rtimes G\:C_0(G)\rtimes_\rt G=\KK\to M(A\rtimes_\alpha G)
\]
is $\what\rt-\what\alpha$ equivariant.
Since the perturbed coaction 
$\ad U\circ\what\rt$ 
is trivial,
it follows that
$\wilde\alpha$ is trivial on the image $(\mu\rtimes G)(\KK)$.
Thus $(A\rtimes_\alpha G,\wilde\alpha,\mu\rtimes G)$ is a \kco.

Given a morphism $\phi$, we know from 
Lemmas~\ref{ac x pr comp} and \ref{x pr d comp}
that $\phi\rtimes G$ is both $(i_A\circ\mu)-(i_B\circ\nu)$ and $i_G^\alpha-i_G^\beta$ compatible, and it follows that it is also $(\mu\rtimes G)-(\nu\rtimes G)$ compatible,
since
$\mu\rtimes G$ is the integrated form of the covariant representation $(i_A\circ\mu,i_G^\alpha)$ of the action $(C_0(G),\rt)$,
and similarly for $\nu\rtimes G$.

Since
$(\phi\rtimes G)\:(A\rtimes_\alpha G,\what\alpha,i_A\circ\mu)\to(B\rtimes_\beta G,\what\beta,i_B\circ\nu)$
is a morphism of
\cgco s,
the $(i_A\circ\mu)-(i_B\circ\nu)$ compatibility, combined with \lemref{mor coc}, also gives $\wilde\alpha-\wilde\beta$ equivariance,
and therefore $\phi\rtimes G$ is a morphism
of \kco s.

Finally, the functoriality is clear, because we know that the crossed-product construction is functorial from actions to coactions.
\end{proof}

\subsection*{Relative commutants}

\begin{rem}
In the following we refer to \cite{fischer} and \cite{koqstable} for relative commutants of the compact operators $\KK=\KK(L^2(G))$. While the first reference seems not to require $G$ to be second countable --- \ie, for the Hilbert space $L^2(G)$ to be separable --- the second reference does assume this. However, the methods of \cite{koqstable} can be extended to the general case using routine methods.
\end{rem}

The
\emph{relative commutant} of a \kalg\ $(A,\iota)$ is the $C^*$-algebra
\[
C(A,\iota)=\{m\in M(A):m\iota(k)=\iota(k)m\in A
\text{ for all }k\in\KK\}.
\]
\begin{prop}[{\cite[Remark~3.1]{fischer}, \cite[Proposition~3.4]{koqstable}}]
If $(A,\iota)$ is a \kalg\ then there is a unique isomorphism
$\theta_A\:C(A,\iota)\xt\KK\variso A$
such that
\[
\theta(m\xt k)=m\iota(k)\righttext{for}m\in C(A,\iota),k\in\KK.
\]
\end{prop}

As Fischer observes, $C(A,\iota)$ can be characterized as the unique closed subset $Z$ of $M(A)$ 
that commutes elementwise with $\iota(\KK)$ 
and 
satisfies
$\clspn \{Z\iota(\KK)\}=A$.
Similarly, $M(C(A,\iota))$ can be characterized as the set of all elements of $M(A)$ that commute 
elementwise
with
$\iota(\KK)$.

\begin{lem}\label{C functor}
Let $\phi\:(A,\iota)\to (B,\jmath)$ be a morphism of \kalg s.
Then there is a unique homomorphism
\[
C(\phi)\:C(A,\iota)\to C(B,\jmath)
\]
such that
\begin{equation}\label{phi km}
\phi(\iota(k)m)=\jmath(k)C(\phi)(m)
\righttext{for all}k\in\KK,m\in C(A,\iota).
\end{equation}
We also have
\begin{equation}\label{phi mk}
\phi(m\iota(k))=C(\phi)(m)\jmath(k)
\righttext{for all}m\in C(A,\iota),k\in\KK.
\end{equation}
Finally, in this way we get a functor
\[
(A,\iota)\mapsto C(A,\iota)
\]
from \kalg s to $C^*$-algebras.
\end{lem}

\begin{proof}
By \lemref{extend}, $\phi$ extends uniquely to a
$\KK$-strictly continuous
homomorphism $\bar\phi\:M_\KK(A)\to M_\KK(B)$.
Moreover, for $k\in \KK$ and $m\in M_\KK(A)$,
\[
\phi(\iota(k)m)=\jmath(k)\bar\phi(m)
\midtext{and}
\phi(m\iota(k))=\bar\phi(m)\jmath(k).
\]
Thus the restriction
\[
C(\phi)=\bar\phi|_{C(A,\iota)}
\]
is a homomorphism
of $C(A,\iota)$ into $M_\KK(B)$
satisfying \eqref{phi km}--\eqref{phi mk}.
We want to know
that
$C(\phi)$ maps into $C(B,\jmath)$:
since
the range of $C(\phi)$ is contained in $M_\KK(B)$,
we only need to observe that if $m\in C(A,\iota)$ and $k\in\KK$, then
\[
\phi(m)\jmath(k)
=\phi(m\iota(k))
=\phi(\iota(k)m)
=\jmath(k)\phi(m),
\]
so $\phi(m)\in C(B,\jmath)$.
The uniqueness
of $\bar\phi$ subject to \eqref{phi km}
is clear since $\jmath$ is nondegenerate.

For the functoriality,
let $\phi\:(A,\iota)\to (B,\jmath)$ and $\psi\:(B,\jmath)\to (C,\eta)$ be morphisms of \kalg s.
Then 
$\psi\circ\phi\:(A,\iota)\to (C,\eta)$ is a morphism
such that
for all $m\in C(A,\iota)$ and $k\in\KK$ we have
\begin{align*}
\psi\circ\phi(m\iota(k))
&=\psi\bigl(\phi(m\iota(k))\bigr)
\\&=\psi\bigl(C(\phi)(m)\jmath(k))\bigr)
\\&=C(\psi)\bigl(C(\phi)(m)\bigr)\eta(k)
\\&=C(\psi)\circ C(\phi)(m)\eta(k),
\end{align*}
so $C(\psi\circ\phi)=C(\psi)\circ C(\phi)$ by uniqueness.
Since identity morphisms pose no problem, we are done.
\end{proof}

It is also useful to have a nondegenerate version
of this functor:

\begin{cor}[{\cite{fischer}}]\label{nd c functor}
\label{C functor nondegenerate}
Let $(A,\iota)$ and $(B,\jmath)$ be \kalg s, and let $\phi\:A\to M(B)$ be a nondegenerate homomorphism such that
\[
\phi\circ\iota=\jmath.
\]
Then the restriction of
\(the canonical extension to $M(A)$ of\)
$\phi$ to $C(A,\iota)$ is a nondegenerate homomorphism to
$M(C(B,\jmath))$.
Moreover, we have the following functoriality properties:
\begin{itemize}
\item
$C(\id_A)=\id_{C(A,\iota)}$, and

\item
if
$(C,\zeta)$ is another \kalg\
and
$\psi\:B\to M(C)$
is a nondegenerate homomorphism such that $\psi\circ\jmath=\zeta$, then
\[
C(\psi)\circ C(\phi)=C(\psi\circ\phi).
\]
\end{itemize}
\end{cor}

This appears in
\cite[discussion preceding Remark~3.2]{fischer},
and we omit the routine proof.
Our primary use of \corref{C functor nondegenerate} is the following
(see \cite[Remark~3.2]{fischer}):
if $(A,\delta,\iota)$ is a \kco,
then 
$\delta$ restricts to 
a coaction, which we also denote by $C(\delta)$, on $C(A,\iota)$.
This uses the identity
\[
C\bigl(A\xt C^*(G),\iota\xt 1\bigr)=C(A,\iota)\xt C^*(G).
\]

We now have a third functor that we want to officially record with a name.

\begin{lem}\label{c functor}
With the above notation,
the assignments
$(A,\delta,\iota)\mapsto (\relcom(A,\iota),\relcom(\delta))$
and
$\phi\mapsto \relcom(\phi)$
give
a functor
$\relcom$
from \kco s to coactions.
\end{lem}

The name $\relcom$ stands for ``relative commutant'',
and
the functor $\relcom$ is the third and final step in the Fischer maximalization process.

\subsection*{Fischer Maximalization Process}
\label{fischer}

The end result of the composition
\[
\xymatrix{
(A,\delta)
\ar@{|->}[d]^{\cpc}
\\
\text{equivariant action on crossed product}
\ar@{|->}[d]^{\cpa}
\\
\text{\kco\ on double crossed product}
\ar@{|->}[d]^{\relcom}
\\
\text{coaction on relative commutant}
}
\]
is the \emph{maximalization} $(A^m,\delta^m)$ of the coaction $(A,\delta)$.
The canonical isomorphism
$\theta_A\:\relcom(A,\iota)\xt\KK\variso A$
can be used to construct a $\delta^m-\delta$ equivariant surjection
$\psi_A\:A^m\to A$ such that
\[
\psi_A\rtimes G\:A^m\rtimes_{\delta^m} G\to A\rtimes_\delta G
\]
is an isomorphism --- this surjection $\psi_A$ is an official part of the maximalization of $(A,\delta)$.
Moreover, maximalization is a functor since each of the steps in its construction is functorial.
If $\phi\:(A,\delta)\to (B,\epsilon)$ is a morphism of coactions,
we 
write $\phi^m$ for
the associated morphism between maximalizations.
The surjection $\psi$ is  
natural 
in the sense that
if $\phi\:(A,\delta)\to (B,\epsilon)$ is a morphism of coactions then
the diagram
\begin{equation}\label{natural}
\xymatrix{
A^m \ar[r]^{\phi^m} \ar[d]_{\psi_A}
&B^m \ar[d]^{\psi_B}
\\
A \ar[r]_-\phi
&B
}
\end{equation}
commutes
(and in fact this determines $\phi$ uniquely since $\psi$ is surjective).

\section{R-coactions}\label{r-co sec}

We begin with the maximal-tensor-product version of \defnref{d mult}.
\begin{defn}\label{tilde}
For $C^*$-algebras $A,D$ we define
\begin{align*}
\wilde M(A\xm D)
&=\{m\in M(A\xm D):
\\&\hspace{.5in}m(1\xt D)\cup (1\xt D)m\subset A\xm D\}.
\end{align*}
\end{defn}
Note that, using the nondegenerate homomorphism
\[
d\mapsto 1\xm d\:D\to M(A\xm D),
\]
we have
\[
\wilde M(A\xm D)=M_D(A\xm D),
\]
where the right-hand side denotes the $D$-multipliers, as in \defnref{d mult}.

We are ready to define R-coactions,
but first we need an appropriate version of the homomorphism $\delta_G$.
Let $\deltagr\:C^*(G)\to M(C^*(G)\xm C^*(G))$
be the integrated form of the unitary representation $s\mapsto s\xm s$ for $s\in G$.

In
the following definition, and many places elsewhere,
we need to combine maximal tensor products and multipliers.
This can be delicate, and in particular we think it prudent to explicitly record the following 
convention.
Given nondegenerate homomorphisms
$\pi\:A\to M(C)$ and $\rho\:B\to M(D)$,
we denote by $\pi\xm\rho$
the associated homomorphism
\[
\pi\xm\rho\:A\xm B\to M(C\xm D),
\]
which is of course also nondegenerate.
Note that, unlike for minimal tensor products,
the maximal tensor product
$M(C)\xm M(D)$ need not embed faithfully in
$M(C\xm D)$,
so our $\pi\xm\rho$ does \emph{not}
necessarily coincide with the
canonical map
\[
A\xm B\to M(C)\xm M(D)
\]
followed by an inclusion into $M(C\xm D)$.
A particular case of our convention arises when one of the maps $\pi,\rho$ is the identity ---
as, for example in the following definition.

\begin{defn}\label{R coaction}
An \emph{R-coaction} of a locally compact group $G$ is
a pair $(A,\delta)$,
where
$A$ is a $C^*$-algebra
and $\delta$ is
an injective nondegenerate homomorphism
\[
\delta\:A\to \wilde M(A\xm C^*(G))
\]
that satisfies the \emph{coaction identity}
\[
(\delta\xm \id)\circ \delta=(\id\xm \deltagr)\circ \delta
\]
and is \emph{coaction-nondegenerate}:
\[
\clspn\{\delta(A)(1\xm C^*(G))\}=A\xm C^*(G).
\]
\end{defn}
Note in particular that $\deltagr$ is an R-coaction.

This is the style of coaction proposed by Raeburn in \cite{rae:representation}.
Since then it has become more common to use
the minimal tensor product $\xt$ rather than $\xm$.
In this paper, the term ``coaction'' by itself will mean the usual ``standard'' coaction; occasionally we may insert the adjective \emph{standard} to avoid any confusion.

The theory of R-coactions is,
in a limited way,
parallel to that of \coaction s (for example, see \lemref{lem injective} below) --- with a few notable omissions (see the discussion following \lemref{lem injective}).
However, we will limit ourselves to only those aspects that we need; we have no reason to recast the entire theory of coactions in terms of R-coactions.

Regarding \defnref{R coaction},
note that, as for \coaction s, coaction-non\-de\-gen\-eracy implies nondegeneracy as a homomorphism into the multiplier algebra.
In fact, again as for \coaction s:
\begin{lem}\label{lem injective}
Let $\delta\:A\to M(A\xm C^*(G))$ be a homomorphism 
satisfying
\begin{enumerate}
\item
$(\delta\xm\id)\circ\delta=(\id\xm\deltagr)\circ\delta$
and

\item
$\clspn\{\delta(A)(1\xm C^*(G))=A\xm C^*(G)$.
\end{enumerate}
Then $\delta$
is nondegenerate and injective, and also maps into
$\wilde M(A\xm C^*(G))$,
and hence
is an $R$-coaction.

Moreover,
letting
\[
\Upsilon\:A\xm C^*(G)\to A\xt C^*(G)
\]
be the canonical surjection from maximal to minimal tensor products,
the homomorphism
$\delta\sify=\Upsilon\circ\delta$
is a \coaction.
\end{lem}

\begin{proof}
First of all, (ii) clearly implies that $\delta$ is nondegenerate and maps into $\wilde M(A\xm C^*(G))$.
Before showing that $\delta$ is injective,
we 
note
that $\delta\sify$ is a \coaction:
the coaction identity 
is a routine diagram chase,
and the coaction-nondegeneracy is obvious.
Thus,
by \cite[Lemma~2.2]{graded}, for example,
$\delta\sify$
is injective,
and hence 
$\delta$ 
is injective.
\end{proof}

\begin{defn}
In the notation of \lemref{lem injective}, we call $\delta\sify$ the \emph{standardization}
of the R-coaction $\delta$.
\end{defn}

In the opposite direction,
we do not know whether
every \coaction\ 
is the standardization of some R-coaction;
it is true for \coaction s associated with Fell bundles,
and in particular when $G$ is discrete
or $A$ is the (full) crossed product of an action of $G$,
and --- crucially for us --- it is true for maximal coactions (see \thmref{main}),
but the general case remains elusive.

Buss and Echterhoff prove in
\cite[Theorem~5.1]{BEmaximality}
that when $G$ is discrete, a coaction $\delta$ of $G$ on $A$ lifts to a homomorphism 
$A\to A\xm C^*(G)$ 
if and only if $\delta$ is maximal.
Thus, it begins to look as though the property of lifting to an R-coaction is unique to maximal coactions; however, if $G$ is nonamenable but $C^*(G)$ is nuclear,
then the canonical coaction of $G$ on $C^*_r(G)$
is not maximal, but
does lift to a homomorphism into $M(A\xm C^*(G))$
(see \cite[Remark~5.3]{BEmaximality}).
We thank Alcides Buss for mentioning \cite[Theorem~5.1 and Remark~5.3]{BEmaximality} to us.

Additionally,
even if we know that a \coaction\ $\delta$ is 
the standardization of
some R-coaction $\epsilon$,
we have no idea whether $\epsilon$ is unique --- 
we cannot rule out the possibility of adding suitable elements of $\ker\Upsilon$ to get a different R-coaction 
with standardization $\delta$.
We formalize these questions:

\begin{q}\label{exist}
Let $(A,\delta)$ be a \coaction.
\begin{enumerate}
\item
Is there an R-coaction $\epsilon$ such that $\epsilon\sify=\delta$?

\item
If the answer to (i) is yes, is $\epsilon$ unique?
\end{enumerate}
\end{q}

A notable example of a construction for \coaction s that is missing for R-coactions is normalization.
As explained in \cite[Example~A.71 and the surrounding discussion]{enchilada},
if $G$ is a nonamenable discrete group then there is no R-coaction on $C^*_r(G)$ such that $\lambda_s\mapsto \lambda_s\xm s$ for $s\in G$
(alternatively, this follows from \cite[Theorem~5.1]{BEmaximality}).
This has numerous negative consequences:
for example,
we do not know whether every R-coaction has a normalization (in the na\"ive sense of what a normalization would mean here);
in particular,
if the canonical R-coaction $s\mapsto s\xm s$ on $C^*(G)$ has a normalization, it will not in general be the regular representation (as it is for the \coaction\ $s\mapsto s\xt s$).
Additionally,
although it is possible to define covariant representations of an R-coaction $(A,\delta)$ on Hilbert space,
if $(\pi,\mu)$ is a covariant representation
then $\pi(a)\mapsto \ad(\mu\xm\id)(w_G)(\pi(a)\xm 1)$ is not necessarily an R-coaction on the image $\pi(A)$ (unlike the case of \coaction s).

For R-coactions we adopt the following conventions:
\begin{defn}\label{equivariant}
If $(A,\delta)$ and $(B,\epsilon)$ are R-coactions,
a 
homomorphism $\phi\:A\to B$ is
\emph{$\delta-\epsilon$ equivariant} if the diagram
\[
\xymatrix{
A \ar[r]^-\delta \ar[d]_\phi
&\wilde M(A\xm C^*(G)) \ar[d]^{\bar{\phi\xm\id}}
\\
B \ar[r]_-\epsilon
&\wilde M(B\xm C^*(G))
}
\]
commutes.
We also say that $\phi\:(A,\delta)\to (B,\epsilon)$ is a \emph{morphism} of R-coactions.
\end{defn}

Note that the existence of the homomorphism $\bar{\phi\xm\id}$
is guaranteed by \lemref{extend}.

The existence of
a category of R-coactions
rests upon
the following lemma.

\begin{lem}\label{tilde functor}
For a fixed $C^*$-algebra $D$,
the assignments $A\mapsto \wilde M(A\xm D)$
and $\phi\mapsto \bar{\phi\xm\id}$
give a functor on the category of $C^*$-algebras.
\end{lem}

\begin{proof}
The assignments $A\mapsto A\xm D$ and $\phi\mapsto \phi\xm\id$ give a functor from $C^*$-algebras to \dalg s,
and then by applying \corref{md funct} we can compose to get the desired functor.
\end{proof}

\begin{cor}
With morphisms as in \defnref{equivariant},
we get a category of R-coactions.
\end{cor}

\begin{proof}
With the aid of \lemref{tilde functor}, it is easy to see that morphisms can be composed, and then associativity and identity morphisms are routinely checked.
\end{proof}

\begin{lem}\label{s functor}
A homomorphism that is equivariant for R-coactions is also equivariant for the associated 
standardizations.
Consequently,
the
assignments $\delta\mapsto \delta\sify$
give
a functor
from 
R-coactions to 
\coaction s.
\end{lem}

\begin{proof}
Let $\phi\:(A,\delta)\to (B,\epsilon)$ be a morphism of R-coactions.
Then the diagram
\[
\xymatrix{
&&M(A\xt C^*(G)) \ar[ddd]^{\phi\xt\id}
\\
A \ar[r]_-\delta \ar[d]_\phi \ar@/^/[urr]^-{\delta\sify}
&M(A\xm C^*(G)) \ar[d]^{\phi\xm\id} \ar[ur]_\Upsilon
\\
B \ar[r]^-\epsilon \ar@/_/[drr]_-{\epsilon\sify}
&M(B\xm C^*(G)) \ar[dr]^\Upsilon
\\
&&M(B\xt C^*(G))
}
\]
commutes,
because the left-hand rectangle commutes by assumption and $\Upsilon$ is natural.
This proves the first part, and then the functoriality follows from a routine computation.
\end{proof}

\begin{rem}
It is natural to wonder whether \lemref{s functor}
is also true in the other direction:
if $(A,\delta)$ and $(B,\epsilon)$ are R-coactions
and
$\phi\:(A,\delta\sify)\to (B,\epsilon\sify)$ is a morphism
of \coaction s,
is $\phi$ is also
$\delta-\epsilon$
equivariant?
We do not know.
This is another way in which the theory of R-coactions is impoverished in comparison to \coaction s,
and consequently is why we do not want to establish any more of the theory concerning R-coactions than we need.
\end{rem}

\begin{rem}
We have defined a category of R-coactions in which the morphisms are homomorphisms between the $C^*$-algebras themselves.
It is possible to define a ``nondegenerate'' version, using nondegenerate homomorphisms into multiplier algebras.
However, since we have in mind no immediate application of this, we eschew it for now.
\end{rem}

\begin{defn}\label{dual R coaction}
The \emph{dual R-coaction} $\direct\alpha$ on the crossed product
$A\rtimes_\alpha G$ of an action $(A,\alpha)$ is defined as the integrated form of the covariant representation
\[
\bigl(i_A\xm 1,(i_G\xm\id)\circ\deltagr\bigr).
\]
\end{defn}

\begin{lem}\label{R alpha}
With the above notation, $\direct\alpha$ really is an R-coaction.
Moreover, the standardization
$\direct\alpha\sify$
of the dual R-coaction
is the usual dual \coaction\ $\what\alpha$.
\end{lem}

\begin{proof}
Both parts follow from routine calculations.
\end{proof}

The following lemma shows
that the dual R-coaction is natural:

\begin{lem}\label{dual R coaction functor}
If $\phi\:(A,\alpha)\mapsto (B,\beta)$ is a morphism of actions
then the crossed product $\phi\rtimes G\:A\rtimes_\alpha G\to B\rtimes_\beta G$
is $\direct\alpha-\direct\beta$ equivariant,
and hence gives a morphism
\[
\phi\rtimes G\:\bigl(A\rtimes_\alpha G,\direct\alpha\bigr)\to \bigl(B\rtimes_\beta G,\direct\beta\bigr).
\]
In this way 
the assignments 
$(A,\alpha)\mapsto (A\rtimes_\alpha G,\direct\alpha)$
give
a functor from actions to R-coactions.
\end{lem}

\begin{proof}
Using functoriality of $(A,\alpha)\mapsto A\rtimes_\alpha G$,
the lemma follows 
once we 
have verified
that
$\phi\rtimes G$ is $\direct\alpha-\direct\beta$ equivariant,
and we check this on generators:
for $a\in A$ we have
\begin{align*}
\direct\beta\circ (\phi\rtimes G)\circ i_A(a)
&=\direct\beta\circ i_B\circ \phi(a)
\\&=i_B\circ \phi(a)\xm 1
\\&=\bigl((\phi\rtimes G)\xm \id\bigr)\bigl(i_A(a)\xm 1\bigr)
\\&=\bigl((\phi\rtimes G)\xm \id\bigr)\circ \direct\alpha\circ i_A(a),
\end{align*}
and for $s\in G$ we have
\begin{align*}
\direct\beta\circ (\phi\rtimes G)\circ i_G^\alpha(s)
&=\direct\beta\circ i_G^\beta(s)
\\&=i_G^\beta(s)\xm s
\\&=\bigl((\phi\rtimes G)\xm \id\bigr)\bigl(i_G^\alpha(s)\xm s\bigr)
\\&=\bigl((\phi\rtimes G)\xm \id\bigr)\circ \direct\alpha\circ i_G(s).
\qedhere
\end{align*}
\end{proof}

It is important to note that
the functor in \lemref{dual R coaction functor}
produces the same $C^*$-algebras $A\rtimes_\alpha G$
and morphisms $\phi\rtimes G$
as for \coaction s
(more precisely, the morphisms involve the same homomorphisms).

\begin{defn}\label{invariant ideal}
Let $(A,\delta)$ be an R-coaction,
let $I$ be an ideal of $A$,
and let $\pi\:A\to A/I$ be the quotient map.
We say that $I$ is \emph{$\delta$-invariant} if 
\begin{equation}\label{contained}
I\subset \ker(\pi\xm \id)\circ\delta.
\end{equation}
\end{defn}

\begin{rem}
Note that it would be sensible to consider another, stronger, notion of $\delta$-invariant ideal $I$ of $A$,
namely that $\delta$ should restrict to give 
an R-coaction
on $I$.
The existence of two different types of invariant ideals
will cause no confusion; in this paper, for R-coactions we will always mean invariance in the sense of \defnref{invariant ideal}.

We do not know (and fortunately do not need to know) whether an ideal is $\delta$-invariant if and only if it is $\delta\sify$-invariant.
\end{rem}

The following is a well-known fact in the context of \coaction s, and we shall need it for R-coactions.

\begin{lem}\label{lem invariant ideal}
Let $(A,\delta)$ be an R-coaction.
Then 
an ideal $I$ of $A$ is $\delta$-invariant
if and only if
there is a 
R-coaction $\epsilon$ on $A/I$ 
such that the quotient map $\pi$ is
$\delta-\epsilon$ equivariant.
Moreover, in this case we actually have equality in \eqref{contained}.
\end{lem}

\begin{proof}
Similarly to 
\coaction s, the condition \eqref{contained}
is exactly what is needed for there to exist a homomorphism $\epsilon$ making the diagram
\[
\xymatrix{
A \ar[r]^-\delta \ar[d]_\pi
&M(A\xm C^*(G)) \ar[d]^{\pi\xm \id}
\\
A/I \ar@{-->}[r]_-\epsilon
&M(A/I\xm C^*(G))
}
\]
commute,
in which case it is routine to verify
that $\epsilon$ satisfies
(the versions for R-coactions of)
the coaction identity and coaction-non\-de\-gen\-er\-acy,
and hence by \lemref{lem injective}
is
an R-coaction.

In particular,
$\epsilon$
is
injective,
so
\[
I=\ker\pi=\ker(\pi\xm\id)\circ\delta.
\qedhere
\]
\end{proof}

\section{R-ification}\label{rification}

Now we will construct a functor,
which we will call R-ification,
from maximal \coaction s to R-coactions.
We start with a maximal \coaction\ $(A,\delta)$.
Even though $\delta$ itself is maximal, we need to first apply (a slightly embellished version of) the maximalization functor to it in order to have sufficient data to make a functorial choice of R-coaction.

In \secref{fischer construction}
we reviewed (for \coaction s) the details of the maximalization functor
as described in
\cite{fischer} and
\cite[Sections~3--4]{koqmaximal}.
In \secref{r-co sec} we began the process of adapting some of this to the context of R-coactions,
and we continue this process in the current section.

Among other things, we will need to adapt to R-coactions a bit of the theory of cocycles.
But first, we need an appropriate version of
the abstract \dco s.

\begin{defn}
Fix a $C^*$-algebra $D$.
A \emph{\dcor} is a triple $(A,\delta,\pi)$,
where $(A,\delta)$ is an R-coaction and $(A,\pi)$ is a 
\dalg\
such that
\[
\delta\circ\pi=\pi\xm 1.
\]
\end{defn}

Note that the categories of R-coactions and of \dalg s combine immediately to form a \emph{category of \dcor s},
where a morphism $\phi\:(A,\delta,\pi)\to (B,\epsilon,\psi)$ is just a morphism
$\phi\:(A,\delta)\to (B,\epsilon)$ of R-coactions
that is also $\pi-\psi$ compatible.

\begin{defn}\label{R-cocycle}
A \emph{cocycle} for an R-coaction $(A,\delta)$,
also called a \emph{$\delta$-cocycle},
is a unitary $U\in M(A\xm C^*(G))$ such that
\begin{enumerate}
\item
$(\id\xm\deltagr)(U)=(U\xm 1)(\delta\xm\id)(U)$
and

\item
$\clspn\{(1\xm C^*(G))U\delta(A)\}=A\xm C^*(G)$.
\end{enumerate}
\end{defn}

\begin{lem}
If $(A,\delta)$ is an R-coaction and
$U$ is a $\delta$-cocycle, then $\epsilon=\ad U\circ\delta$ is also an R-coaction on $A$.
\end{lem}

\begin{proof}
The proof is a routine modification of the argument for \coaction s;
we include it here for completeness.
We will appeal to \lemref{lem injective}:
certainly
\[
\epsilon\:A\to M(A\xm C^*(G))
\]
is a homomorphism.
We check hypotheses~(i)--(ii) in \lemref{lem injective}:
\begin{align*}
(\epsilon\xm\id)\circ\epsilon
&=(\ad U\circ\delta\xm\id)\circ \ad U\circ\delta
\\&=\ad (U\xm 1)\circ(\delta\xm\id)\circ \ad U\circ\delta
\\&=\ad (U\xm 1)\circ \ad (\delta\xm\id)(U)\circ(\delta\xm\id)\circ\delta
\\&=\ad \bigl((U\xm 1)(\delta\xm\id)(U)\bigr)\circ(\id\xm\deltag^R)\circ\delta
\\&=\ad (\id\xm\deltag^R)(U)\circ(\id\xm\deltagr)\circ\delta
\\&=(\id\xm\deltagr)\circ \ad U\circ\delta
\\&=(\id\xm\deltagr)\circ \epsilon,
\end{align*}
and
\begin{align*}
\clspn \bigl\{\epsilon(A)(1\xm C^*(G))\bigr\}
&=\clspn \bigl\{\ad U\circ\delta(A)(1\xm C^*(G))\bigr\}
\\&=\ad U\Bigl(\clspn \bigl\{\delta(A)(1\xm C^*(G))\bigr\}\Bigr)
\\&=\ad U(A\xm C^*(G))
\\&=A\xm C^*(G).
\qedhere
\end{align*}
\end{proof}

\begin{defn}
With the above notation, we say that $\epsilon$ is \emph{exterior equivalent} to $\delta$,
and also we say that it is a \emph{perturbation} of $\delta$.
\end{defn}

\begin{lem}
Exterior equivalence
is an equivalence relation on 
R-coactions.
\end{lem}

\begin{proof}
Trivially every R-coaction is exterior equivalent to itself (using the cocycle 1 (the identity element of the appropriate multiplier algebra).
Symmetry and transitivity follow from
the following facts:
given an R-coaction $(A,\delta)$,
if $U$ is a $\delta$-cocycle,
then $U^*$ is an $\ad U\circ\delta$-cocycle,
and if $W$ is an $\ad U\circ\delta$-cocycle,
then $WU$ is a $\delta$-cocycle.
Both of these follow from routine computations,
just as for \coaction s.
\end{proof}

Note that condition (ii) in \defnref{R-cocycle} guarantees that $\epsilon$ is coaction-nondegenerate.
This is one of those places where we give only a limited development of the theory, to avoid tedious, irrelevant discussions:
for \coaction s, the usual definition of cocycle has the following condition instead of (ii):
\begin{enumerate}
\item[(ii)$'$]
$\ad U\circ\delta(A)(1\xm C^*(G))\subset A\xm C^*(G)$,
\end{enumerate}
and then it is remarked that the \coaction\ $\ad U\circ\delta$ is automatically nondegenerate because $\delta$ is.
However, we do not want to spend the effort trying to prove the corresponding fact for R-coactions, so we assume (ii) above, which implies (ii)$'$.
This seems to us to be a reasonable compromise, since we will not need to know very many properties of cocycles for R-coactions.

\begin{lem}\label{nd morphism}
Let $(A,\delta)$ and $(B,\epsilon)$ be R-coactions,
let $\phi\:A\to M(B)$ be a nondegenerate $\delta-\epsilon$ equivariant homomorphism,
and let $U$ be a $\delta$-cocycle.
Then $(\phi\xm\id)(U)$ is an $\epsilon$-cocycle.
\end{lem}

\begin{proof}
The proof is a routine modification of the argument for \coaction s;
we give it for completeness:
\begin{align*}
&(\id\xm\deltagr)\bigl((\phi\xm\id)(U)\bigr)
\\&\quad=(\phi\xm\id\xm\id)\bigl((\id\xm\deltagr)(U)\bigr)
\\&\quad=(\phi\xm\id\xm\id)\bigl((U\xm 1)(\delta\xm\id)(U)\bigr)
\\&\quad=\bigl((\phi\xm\id)(U)\xm 1\bigr)\bigl((\phi\xm\id)\circ\delta\xm\id\bigr)(U)
\\&\quad=\bigl((\phi\xm\id)(U)\xm 1\bigr)(\epsilon\xm\id)\bigl((\phi\xm\id)(U)\bigr),
\end{align*}
and
\begin{align*}
&\clspn\bigl\{(1\xm C^*(G))(\phi\xm\id)(U)\epsilon(B)\bigr\}
\\&\quad=\clspn\bigl\{(1\xm C^*(G))(\phi\xm\id)(U)\epsilon\bigl(\phi(A)B\bigr)\bigr\}
\\&\quad=\clspn\bigl\{(1\xm C^*(G))(\phi\xm\id)(U)(\phi\xm\id)\circ\delta(A)\epsilon(B)\bigr\}
\\&\quad=\clspn\bigl\{(\phi\xm\id)\Bigl(\bigl(1\xm C^*(G)\bigr)U\delta(A)\Bigr)\epsilon(B)\bigr\}
\\&\quad=\clspn\bigl\{(\phi\xm\id)\Bigl(
\clspn\bigl\{\bigl(1\xm C^*(G)\bigr)U\delta(A)\bigr\}
\Bigr)\epsilon(B)\bigr\}
\\&\quad=\clspn\bigl\{(\phi\xm\id)\Bigl(
\bigl(A\xm C^*(G)\bigr)
\Bigr)\epsilon(B)\bigr\}
\\&\quad=\clspn\bigl\{\bigl(\phi(A)\xm C^*(G)\bigr)\epsilon(B)\bigr\}
\\&\quad=\clspn\bigl\{\bigl(\phi(A)\xm 1\bigr)\bigl(1\xm C^*(G)\bigr)\epsilon(B)\bigr\}
\\&\quad=\clspn\bigl\{\bigl(\phi(A)\xm 1\bigr)
\clspn\bigl\{\bigl(1\xm C^*(G)\bigr)\epsilon(B)\bigr\}
\bigr\}
\\&\quad=\clspn\bigl\{\bigl(\phi(A)\xm 1\bigr)
\bigl(B\xm C^*(G)\bigr)
\bigr\}
\\&\quad=B\xm C^*(G).
\qedhere
\end{align*}
\end{proof}

The following is a version of
\lemref{fn coc} for R-coactions:

\begin{cor}\label{fn coc r}
If $(A,\delta,\mu)$ is a \cgcor,
then
$W=(i_A\circ\mu\xm\id)(w_G)$
is 
a $\delta$-cocycle.
\end{cor}

\begin{proof}
This is a routine adaptation of the argument for 
\propref{fn coc}, replacing minimal tensor products by maximal ones.
Note that the trivial coaction $\triv$ on $C_0(G)$ may be regarded as an R-coaction,
$w_G$ is still a cocycle for $\triv$,
and $\mu$ is $\triv-\delta$ equivariant.
Thus we can appeal to \lemref{nd morphism}.
\end{proof}

\begin{defn}
As we did for \coaction s, in \corref{fn coc r} we call $W$ the \emph{associated cocycle}.
\end{defn}

\begin{lem}\label{mor coc r}
If $\phi\:(A,\delta,\mu)\to (B,\epsilon,\nu)$ is a morphism of \cgcor s,
with associated cocycles $W$ and $U$, respectively,
then $\phi$ is also $\ad W\circ\delta-\ad U\circ\epsilon$ equivariant.
\end{lem}

\begin{proof}
For \coaction s, this result is
\lemref{mor coc},
which in turn is
contained in \cite[Lemma~3.6]{klqfunctor2};
we cannot merely adapt the proof found there, however,
because it relies in part upon the following embedding property of minimal tensor products:
if $D\subset B$ then $D\xt C^*(G)\subset B\xt C^*(G)$,
which does not carry over to maximal tensor products.
So, we have to make a completely new proof.

We must show 
that
the diagram
\begin{equation}\label{cocycle morphism diagram}
\xymatrix@C+70pt{
A \ar[r]^-{\ad (\mu\xm\id)(w_G)\circ\delta} \ar[d]_\phi
&\wilde M(A\xm C^*(G)) \ar[d]^{\phi\xm\id}
\\
B \ar[r]_-{\ad (\nu\xm\id)(w_G)\circ\epsilon}
&\wilde M(B\xm C^*(G))
}
\end{equation}
commutes.
It suffices to show that
for all $a\in A$ we have
\begin{equation}\label{phi w}
(\phi\xm\id)\bigl((\mu\xm\id)(w_G)\delta(a)\bigr)=(\nu\xm\id)(w_G)\epsilon(\phi(a)),
\end{equation}
because we can take the adjoint of both sides, then multiply, and use the fact that every element of $A$ can be factored in the form $ab^*$ for some $a,b\in A$.

It further suffices to show that
\begin{equation}\label{abstract}
(\phi\xm\id)\bigl((\mu\xm\id)(x)y\bigr)=(\nu\xm\id)(x)(\phi\xm\id)(y)
\end{equation}
holds
for all
$x\in M(C_0(G)\xm C^*(G))$
and
$y\in \wilde M(A\xm C^*(G))$
and we will accomplish this in a sequence of steps.

We first show that \eqref{abstract} holds for all $x\in C_0(G)\xm C^*(G)$
(note that $\xm=\xt$ here)
and $y\in A\xm C^*(G)$:
by linearity, continuity, and density (for the norm topologies),
it suffices to consider elementary tensors
$x=f\xm c$ and $y=a\xm d$
for $f\in C_0(G)$, $a\in A$, and $c,d\in C^*(G)$:
\begin{align*}
&(\phi\xm\id)\bigl((\mu\xm\id)(f\xm c)(a\xm d)\bigr)
\\&\quad=(\phi\xm\id)\bigl(\mu(f)a\xm cd\bigr)
\\&\quad=\phi\bigl(\mu(f)a\bigr)\xm cd
\\&\quad=\nu(f)\phi(a)\xm cd
\\&\quad=(\nu(f)\xm c)\bigl(\phi(a)\xm d\bigr)
\\&\quad=(\nu\xm\id)(f\xm c)(\phi\xm\id)(a\xm d).
\end{align*}

We will now deduce that \eqref{abstract} holds for any 
$x\in M(C_0(G)\xm C^*(G))$
and $y\in A\xm C^*(G)$,
using norm continuity of $\phi\xm\id$:
choose a net $(x_i)$ in $C_0(G)\xm C^*(G)$ converging strictly to $x$ in $M(C_0(G)\xm C^*(G))$,
and let $y\in A\xm C^*(G)$.
Then $(\mu\xm\id)(x_i)\to (\mu\xm\id)(x)$ strictly in $M(A\xm C^*(G))$
because $\mu$ is nondegenerate,
and so
\[
(\mu\xm\id)(x_i)y\to (\mu\xm\id)(x) y\righttext{in norm.}
\]
Since $\phi\xm\id$ is norm continuous,
\[
(\phi\xm\id)\bigl((\mu\xm\id)(x_i)y\bigr)\to (\phi\xm\id)((\mu\xm\id)(x) y).
\]
For each $i$,
\[
(\phi\xm\id)\bigl((\mu\xm\id)(x_i)y\bigr)=(\nu\xm\id)(x_i)(\phi\xm\id)(y).
\]
Since $\nu\xm\id\:M(C_0(G)\xm C^*(G))\to M(B\xm C^*(G))$ is strictly continuous,
\[
(\nu\xm\id)(x_i)\to (\nu\xm\id)(x)\righttext{strictly.}
\]
Since $(\phi\xm\id)(y)\in B\xm C^*(G)$,
\[
(\nu\xm\id)(x_i)(\phi\xm\id)(y)
\to (\nu\xm\id)(x)(\phi\xm\id)(y)
\]
in norm.
Combining all the above, we get \eqref{abstract} for all 
$x\in M(C_0(G)\xm C^*(G))$
and $y\in A\xm C^*(G)$,
as desired.

Finally, we take an arbitrary
$x\in M(C_0(G)\xm C^*(G))$
and
$y\in \wilde M(A\xm C^*(G))$.
Letting $(e_i)$ be an approximate identity for $C^*(G)$,
we have $y(1\xm e_i)\in A\xm C^*(G)$ for all $i$,
so by the above we obtain
\begin{align*}
&(\phi\xm\id)\bigl((\mu\xm\id)(x)y\bigr)(1_{M(B)}\xm e_i)
\\&\quad=(\phi\xm\id)\bigl((\mu\xm\id)(x)y(1_{M(A)}\xm e_i)\bigr)
\\&\quad=(\nu\xm\id)(x)(\phi\xm\id)\bigl(y(1_{M(A)}\xm e_i)\bigr)
\\&\quad=(\nu\xm\id)(x)(\phi\xm\id)(y)(1_{M(B)}\xm e_i).
\end{align*}
Since 
$(1_{M(B)}\xm e_i)$
converges $C^*(G)$-strictly to 1 in 
$\wilde M(B\xm C^*(G))$,
we conclude that
\[
(\phi\xm\id)\bigl((\mu\xm\id)(x)y\bigr)
=(\nu\xm\id)(x)(\phi\xm\id)(y).
\qedhere
\]
\end{proof}

We need the following version for R-coactions of
\lemref{cpa functor},
and the proof is almost the same.

\begin{lem}\label{tilde alpha}
Let $(A,\alpha,\mu)$ be an equivariant action.
Then
\[
W=(i_A\circ\mu\xm\id)(w_G)
\]
is an $\direct\alpha$-cocycle.
Put
\[
\perturb\alpha=\ad W\circ\direct\alpha.
\]
Then
$(A\rtimes_\alpha G,\wilde\alpha,\mu\rtimes G)$ is a \kcor.

Moreover, if $\phi\:(A,\alpha,\mu)\to (B,\beta,\nu)$ is a morphism of equivariant actions,
then
the homomorphism
$\phi\rtimes G$
gives
a morphism
\[
(A\rtimes_\alpha G,\wilde\alpha,\mu\rtimes G)\to
(B\rtimes_\beta G,\wilde\beta,\nu\rtimes G)
\]
of \kcor s.

Further, the assignments
$(A,\alpha,\mu)\mapsto (A\rtimes_\alpha G,\wilde\alpha,\mu\rtimes G)$
and
$\phi\mapsto \phi\rtimes G$
give
a functor
$\cpar$
from equivariant actions to \kcor s.

Finally, we have
$\perturb\alpha\sify=\wilde\alpha$.
\end{lem}

In the statement of the above lemma, the name ``$\cpar$'' is intended to indicate that we have modified the functor $\cpa$ so that the output is an R-coaction rather than a standard coaction.

\begin{proof}
Similarly to the argument for \lemref{cpa functor},
$\mu\rtimes G$ is equivariant for the R-coactions
$\direct\rt$ and $\direct\alpha$.
Since $(M\xm\id)(w_G)$ is an $\direct\rt$-cocycle,
the unitary $W$
is an $\direct\alpha$-cocycle.
Since $\ad (M\xm\id)(w_G)\circ \direct\rt$ is trivial,
the R-coaction $\perturb\alpha$ is trivial on $(\mu\rtimes G)(\KK)$,
and so we get a \kcor\
$(A\rtimes_\alpha G,\perturb\alpha,\mu\rtimes G)$.

Given a morphism $\phi$,
we know already from the proof of
\lemref{cpa functor}
that $\phi\rtimes G$ is $(\mu\rtimes G)-(\nu\rtimes G)$ compatible.

Since
$(\phi\rtimes G)\:(A\rtimes_\alpha G,\what\alpha,i_A\circ\mu)\to(B\rtimes_\beta G,\what\beta,i_B\circ\nu)$
is a morphism of
\cgco s,
the $(i_A\circ\mu)-(i_B\circ\nu)$ compatibility, combined with 
\lemref{mor coc r},
also gives $\perturb\alpha-\perturb\beta$ equivariance,
and therefore $\phi\rtimes G$ is a morphism
of \kcor s.

The functoriality is clear, because we know that the crossed-product construction is functorial from actions to coactions.

Finally,
the last assertion, concerning the standardization, follows from a routine computation on the generators.
\end{proof}

\begin{lem}\label{c of r co}
If $(A,\delta,\iota)$ is a \kcor,
then $\delta$ restricts to an R-coaction $\relcom(\delta)$ on the relative commutant $\relcom(A,\iota)$.
Moreover, the assignments $(A,\delta,\iota)\mapsto (\relcom(A,\iota),\relcom(\delta))$
and $\phi\mapsto \relcom(\phi)$
give a functor
$\c-r$
from \kcor s to R-coactions.

Moreover, this functor respects standardizations,
\ie,
\[
\relcom(\delta)\sify=\relcom(\delta\sify).
\]
\end{lem}

In the above, the name $\c-r$ is meant to indicate ``commutant of R-coaction''.

For the proof, we need to know how tensoring with a fixed $C^*$-algebra affects relative commutants.

\begin{lem}\label{C tensor D}
Let $(A,\iota)$ be a \kalg, and let $D$ be any $C^*$-algebra.
Then we have another \kalg\ $(A\xm D,\iota\xm 1_D)$.
Define an isomorphism $\sigma$ by the commutative diagram
\[
\xymatrix@C+30pt{
A\xm D \ar[dr]_(.4)\sigma^\simeq \ar[r]^-{\theta_A\inv\xm\id}_-\simeq
&\relcom(A,\iota)\xm\KK\xm D \ar[d]^{\id\xm\Sigma}_\simeq
\\
&\relcom(A,\iota)\xm D\xm\KK.
}
\]
Then there is a unique isomorphism
\[
\tau\:\relcom(A\xm D,\iota\xm 1_D)\variso \relcom(A,\iota)\xm D
\]
such that
\[
\sigma(m)=\tau(m)\xm 1_\KK\righttext{for all}m\in \relcom(A\xm D,\iota\xm 1_D).
\]
\end{lem}

When using \lemref{C tensor D},
we will usually suppress the $\tau$ and identify
\[
\relcom(A\xm D,\iota\xm 1_D)=\relcom(A,\iota)\xm D.
\]

In turn, for the proof of the auxiliary \lemref{C tensor D}
it will be convenient to have available
the following fact, which gives a convenient tool to recognize relative commutants.

\begin{lem}\label{lem recognize C}
Let $(A,\iota)$ be a \kalg\ and $B$ a $C^*$-algebra,
and
suppose that $\sigma\:A\variso B\xt\KK$ is an isomorphism
such that $\sigma\circ\iota=1_B\xt \id_\KK$.
Then
there is a unique isomorphism $\tau\:\relcom(A,\iota)\variso B$ such that
\[
\sigma(m)=\tau(m)\xt 1_\KK\righttext{for}m\in \relcom(A,\iota).
\]
\end{lem}

\begin{proof}
It is an immediate consequence of the category equivalence between \kalg s and stable $C^*$-algebras \cite[Theorem~4.4]{koqstable}
that
any isomorphism between \kalg s restricts to an isomorphism between the relative commutants,
and the lemma follows.
\end{proof}

\begin{proof}[Proof of \lemref{C tensor D}]
The isomorphism
$\sigma$
takes $\iota\xm 1_D$ to
the homomorphism
$1_{A\xm D}\xm \id_\KK$.
Consequently,
the result follows from
\lemref{lem recognize C}.
\end{proof}

We are now ready to prove that we can restrict R-coactions to relative commutants.

\begin{proof}[Proof of \lemref{c of r co}]
For the first part, similarly to the proof for $S$-coactions,
we appeal to \lemref{C functor nondegenerate}.
We have a nondegenerate homomorphism
\[
\delta\:A\to M(A\xm C^*(G))
\]
such that
\[
\delta\circ\iota=\iota\xm 1,
\]
so applying
\lemref{C functor nondegenerate}
gives
\[
\relcom(\delta)\:\relcom(A,\iota)\to \relcom(A\xm C^*(G),\iota\xm 1).
\]
By the abstract \lemref{C tensor D},
\[
\relcom\bigl(A\xm C^*(G),\iota\xm 1\bigr)=\relcom(A,\iota)\xm C^*(G),
\]
so we have a nondegenerate homomorphism
\[
\relcom(\delta)\:(A,\iota)\to M\bigl(\relcom(A,\iota)\xm C^*(G)\bigr),
\]
which, by construction, is the restriction of $\delta$ to the $C^*$-subalgebra $\relcom(A,\iota)$ of $M(A)$.

We will show that $(\relcom(A,\iota),\relcom(\delta))$ is an R-coaction.
As we mentioned before, this is done for \coaction s in \cite[Section~3]{fischer} (see \cite[Lemma~3.2]{koqmaximal} for a more detailed proof),
so it might seem tempting to just say something along the lines of ``it carries over routinely to R-coactions''. However, maximal tensor products have hidden subtleties, so it is safer to show (at least some of) the details. We essentially follow the strategy of \cite{koqmaximal}, but we streamline things a bit with the help of the observations we made at the beginning of this proof.

By \lemref{lem injective}, it suffices to
verify the coaction identity and coaction-nondegeneracy.
For the coaction identity, we have
\begin{align*}
&(\relcom(\delta)\xm\id)\circ \relcom(\delta)
\\&\quad=\relcom\bigl((\delta\xm\id)\circ\delta\bigr)
&&\text{(\lemref{C functor nondegenerate})}
\\&\quad=\relcom\bigl((\id\xm\deltag)\circ\delta\bigr)&&\text{(coaction identity for $\delta$)}
\\&\quad=\relcom(\id\xm\deltag)\circ \relcom(\delta)
&&\text{(\lemref{C functor nondegenerate} again).}
\end{align*}
But since $\relcom(\id\xm\deltag)$ is just the restriction of $\id\xm\deltag$ to a $C^*$-subalgebra of
$M(A\xm C^*(G))$,
the above computation tells us that for all $a\in A$ we have
\[
(\relcom(\delta)\xm\id)\circ \relcom(\delta)(a)=(\id\xm\deltag)\circ \relcom(\delta)(a),
\]
proving the coaction identity.

Next, for the coaction-nondegeneracy of $\relcom(\delta)$,
we must show that
\[
\clspn\{\relcom(\delta)(\relcom(A,\iota))(1\xm C^*(G))\}=\relcom(A,\iota)\xm C^*(G),
\]
and we use the following characterization (see \cite[Remark~3.1]{fischer}):
for any \kalg\ $(B,\jmath)$,
the relative commutant $\relcom(B,\jmath)$ is the unique norm-closed subset $Z$ of $M(B)$ satisfying
\begin{enumerate}
\item
$z\jmath(k)=\jmath(k)z$ for all $z\in Z,k\in\KK$, and

\item
$\clspn\{Z\jmath(\KK)\}=B$.
\end{enumerate}
It therefore suffices to verify (i)--(ii) for
the \kalg\
\[
(A\xm C^*(G),\iota\xm 1)
\]
and the closed subset
\[
Z=\clspn\{\relcom(\delta)(\relcom(A,\iota))(1\xm C^*(G))\}.
\]
Obviously
$1\xm C^*(G)$ commutes with $\iota(\KK)\xm 1$,
and so does $\relcom(\delta)(\relcom(A,\iota))$ because for any morphism $\phi\:(B,\jmath)\to (\relcom,\zeta)$ of \kalg s
we know that $\relcom(\phi)(\relcom(B,\jmath))$ commutes with $\zeta(\KK)$.
This implies (i).

For (ii), we note that with the above notation we also have
\begin{equation}\label{eq span}
\clspn\{\relcom(\phi)(\relcom(B,\jmath))\zeta(\KK)\}=\phi(B),
\end{equation}
so
\begin{align*}
&\clspn\bigl\{Z(\iota(\KK)\xm 1)\bigr\}
\\&\quad=\clspn\bigl\{\relcom(\delta)(\relcom(A,\iota))(1\xm C^*(G))(\iota(\KK)\xm 1)\bigr\}
\\&\quad=\clspn\bigl\{\relcom(\delta)(\relcom(A,\iota))(\iota(\KK)\xm 1)(1\xm C^*(G))\bigr\}
\\&\quad\overset{*}=\clspn\bigl\{\delta(A)(1\xm C^*(G))\bigr\}
\\&\quad=A\xm C^*(G),
\end{align*}
where the equality at * follows by applying \eqref{eq span} to the morphism
$\delta\:(A,\iota)\to (A\xm C^*(G),\iota\xm 1)$
of \kalg s.
Thus we have shown the coaction-nondegeneracy, so we have an R-coaction $\relcom(\delta)$ on $\relcom(A,\iota)$.

Finally, we turn to the functoriality:
Since we already know that $(A,\iota)\mapsto \relcom(A,\iota)$ is functorial,
it only remains to verify that 
$\relcom(\phi)$ is $\relcom(\delta)-\relcom(\epsilon)$ equivariant
whenever $\phi\:(A,\delta,\iota)\to (B,\epsilon,\jmath)$ 
is a morphism of \kcor s,
\ie, the diagram
\begin{equation}\label{C phi}
\xymatrix{
\relcom(A,\iota) \ar[r]^-{\relcom(\delta)} \ar[d]_{\relcom(\phi)}
&\wilde M(\relcom(A,\iota)\xm C^*(G)) \ar[d]^-{\relcom(\phi)\xm\id}
\\
\relcom(B,\jmath) \ar[r]_-{\relcom(\epsilon)}
&\wilde M(\relcom(B,\jmath)\xm C^*(G))
}
\end{equation}
commutes.
This might seem almost trivial:
by assumption, the diagram
\[
\xymatrix{
A \ar[r]^-\delta \ar[d]_{\phi}
&\wilde M(A\xm C^*(G)) \ar[d]^{\phi\xm\id}
\\
B \ar[r]_-\epsilon
&\wilde M(B\xm C^*(G))
}
\]
commutes,
and in \eqref{C phi} we seem to be restricting to appropriate subalgebras of the multiplier algebras.
There is a subtlety, however:
the homomorphisms $\phi$ and $\phi\xm\id$ may be degenerate,
so we cannot just
``extend and restrict'' them.
Consequently, we are forced to get our hands dirty:
let $m\in \relcom(A,\iota)$.
We must show that
\[
\epsilon\bigl(\relcom(\phi)(m)\bigr)
=\bigl(\relcom(\phi)\xm\id\bigr)\bigl(\delta(m)\bigr),
\]
in which we have dropped the ``$C$'' for $\delta$ and $\epsilon$, since they are nondegenerate homomorphisms,
and hence for them we are just extending and restricting as usual.
It suffices to show that if we take any $k\in\KK$, then
\[
\epsilon\bigl(\relcom(\phi)(m)\bigr)\bigl(\jmath(k)\xm 1\bigr)
=\bigl(\relcom(\phi)\xm\id\bigr)\bigl(\delta(m)\bigr)\bigl(\jmath(k)\xm 1\bigr).
\]
The computation goes as follows:
\begin{align*}
&\epsilon\bigl(\relcom(\phi)(m)\bigr)\bigl(\jmath(k)\xm 1\bigr)
\\&\quad=\epsilon\bigl(\relcom(\phi)(m)\jmath(k)\bigr)&&\text{($\epsilon\circ\jmath=\jmath\xm 1$)}
\\&\quad=\epsilon\bigl(\phi(m\iota(k))\bigr)&&\text{(Equation~\eqref{phi km})}
\\&\quad=\bar{\phi\xm\id}\bigl(\delta(m\iota(k))\bigr)&&\text{($\phi$ is equivariant)}
\\&\quad=\bar{\phi\xm\id}\bigl(\delta(m)(\iota(k)\xm 1)\bigr)&&\text{($\delta\circ\iota=\iota\xm 1$)}
\\&\quad\overset *=\bar{\relcom(\phi)\xm\id}\bigl(\delta(m)\bigr)(\jmath(k)\xm 1),
\end{align*}
where to justify the equality at * we multiply on the right by $1\xm c$ for an arbitrary $c\in C^*(G)$:
\begin{align*}
&\bar{\phi\xm\id}\bigl(\delta(m)(\iota(k)\xm 1)\bigr)(1\xm c)
\\&\quad\overset 1=(\phi\xm\id)\bigl(\delta(m)(\iota(k)\xm 1)(1\xm c)\bigr)
\\&\quad=(\phi\xm\id)\bigl(\delta(m)(1\xm c)(\iota(k)\xm 1)\bigr)
\\&\quad\overset 2=\bigl(\relcom(\phi)\xm\id\bigr)\bigl(\delta(m)(1\xm c)\bigr)(\jmath(k)\xm 1)
\\&\quad\overset 3=\bar{\relcom(\phi)\xm\id}\bigl(\delta(m)\bigr)(1\xm c)(\jmath(k)\xm 1)
\\&\quad=\bar{\relcom(\phi)\xm\id}\bigl(\delta(m)\bigr)(\jmath(k)\xm 1)(1\xm c),
\end{align*}
and where we used Equation~\eqref{phi km} three times:
for
\begin{itemize}
\item
$\phi\xm\id$ and the embeddings of $C^*(G)$
at equality 1,

\item
$\phi\xm\id$ and the embeddings of $\KK$
at equality 2,
and

\item
$\relcom(\phi)\xm\id$ and the embeddings of $C^*(G)$
at equality 3.
\end{itemize}

For the assertion regarding standardizations,
it is tempting to just say something along the following lines:
for both R-coactions and \coaction s,
the coaction on the relative commutant $\relcom(A,\iota)$
is produced via the nondegenerate version of the relative-commutant functor of \corref{nd c functor}.
But this is really just intuition, not a proof.
Here are the details, for which we will use functoriality of the relative commutant for nondegenerate homomorphisms.
Consider the diagrams
\begin{equation}\label{top diag}
\xymatrix{
A \ar[r]^-\delta \ar[dr]_(.4){\delta\sify}
&M(A\xm C^*(G)) \ar[d]^\Upsilon
\\
&M(A\xt C^*(G))
}
\end{equation}
and
\begin{equation}\label{bottom diag}
\xymatrix{
\relcom(A,\iota) \ar[r]^-{\relcom(\delta)} \ar[dr]_(.4){\relcom(\delta)\sify}
&M(\relcom(A,\iota)\xm C^*(G)) \ar[d]^\Upsilon
\\
&M(\relcom(A,\iota)\xt C^*(G)),
}
\end{equation}
which both commute by definition of standardization.
We will show that \eqref{bottom diag} in fact is the image of \eqref{top diag} under the relative-commutant functor,
and in particular $\relcom(\delta)\sify=\relcom(\delta\sify)$.
The key is to compare the right-hand vertical arrows,
and to recall that both $\relcom(\delta)$ and $\relcom(\delta\sify)$ are images
under the functor $\relcom$
of
morphisms that are nondegenerate as $C^*$-homomorphisms.
Recall that by \lemref{C tensor D} we have
\[
\relcom(A\xm C^*(G),\iota\xm 1)=\relcom(A,\iota)\xm C^*(G),
\]
and similarly for $A\xt C^*(G)$.
Then $\delta$ may be regarded as a \kalg\ morphism
\[
\relcom(A,\iota)\to M(A\xm C^*(G),\iota\xm 1),
\]
where on the right-hand side we really mean the multiplier algebra $M(A\xm C^*(G))$
together with the nondegenerate homomorphism
\[
\iota\xm 1\:\KK\to M(A\xm C^*(G)).
\]
Similarly, $\Upsilon$ is also a \kalg\ morphism:
\[
\Upsilon\:(A\xm C^*(G),\iota\xm 1)\to (A\xt C^*(G),\iota\xt 1),
\]
which, being a surjective homomorphism, can of course be uniquely extended to the multipliers.
Thus \eqref{top diag} is really a commutative diagram in the nondegenerate category of \kalg s:
\[
\xymatrix{
(A,\iota) \ar[r]^-\delta \ar[dr]_(.4){\delta\sify}
&M(A\xm C^*(G),\iota\xm 1) \ar[d]^\Upsilon
\\
&M(A\xt C^*(G),\iota\xt 1).
}
\]
The nondegenerate relative-commutant functor takes this to a commutative diagram
\begin{equation}\label{C diag}
\xymatrix{
\relcom(A,\iota) \ar[r]^-{\relcom(\delta)} \ar[dr]_(.4){\relcom(\delta\sify)}
&M(\relcom(A\xm C^*(G),\iota\xm 1)) \ar[d]^{\relcom(\Upsilon)}
\\
&M(\relcom(A\xt C^*(G),\iota\xt 1)).
}
\end{equation}
Identifying the $C^*$-algebras via the canonical isomorphisms
\begin{align*}
\relcom(A\xm C^*(G),\iota\xm 1)&\simeq \relcom(A,\iota)\xm C^*(G)
\\
\relcom(A\xt C^*(G),\iota\xt 1)&\simeq \relcom(A,\iota)\xt C^*(G),
\end{align*}
we see that the diagonal arrows in the diagrams\eqref{bottom diag} and \eqref{C diag} coincide, as desired.
\end{proof}

\begin{rem}
We pause to contrast the equivariance argument in the above proof with the corresponding argument for \coaction s in the proof of \cite[Lemma~3.8]{klqfunctor2}, where we were able to use the $B(G)$-module structure associated with a \coaction;
we could not use that tactic here because
slice maps do not generally separate the points of maximal tensor products.
\end{rem}

\begin{defn}\label{R of max}
Let $(A,\delta)$ be a maximal coaction.
The
composition
$\c-r\circ\cpar\circ\cpc$ produces an R-coaction $(A^m,\eta)$.
The canonical surjection $\psi_A\:A^m\to A$
is an isomorphism because $\delta$ is maximal,
and so it transforms $\eta$ to an R-coaction on $A$,
which we call the \emph{R-ification} of $\delta$
and denote by $\delta\rify$.
\end{defn}

\begin{thm}\label{main}
The above
R-ification process
$(A,\delta)\mapsto (A,\delta\rify)$
is
a functor
from maximal \coaction s to R-coactions,
where the morphisms are left unchanged: $\phi\mapsto \phi$.
Moreover, we have $\delta\rify\sify=\delta$.
\end{thm}

\begin{proof}
The only nonobvious part is
the functoriality of the last step:
from start to finish we do have a well-defined process
converting
a maximal \coaction\ $\delta$ 
into
an R-coaction $\delta\rify$
on the same $C^*$-algebra $A$.
We need to know that if $\phi\:(A,\delta)\to (B,\epsilon)$ is a morphism of maximal \coaction s, then
the same homomorphism
$\phi\:A\to B$ is also $\delta\rify-\epsilon\rify$ equivariant.
Following the process up until the penultimate step,
\ie, the composition $\c-r\circ\cpar\circ\cpc$,
gives
R-coactions $\eta$ and $\zeta$ on $A^m$ and $B^m$, respectively,
and the morphism $\phi$ is taken to the maximalization
$\phi^m\:A^m\to B^m$,
as we discussed above.
Since
the diagram
\[
\xymatrix{
A^m \ar[r]^-{\phi^m} \ar[d]_{\psi_A}
&B^m \ar[d]^{\psi_B}
\\
A \ar[r]_-\phi
&B
}
\]
commutes,
we see that $\phi\:A\to B$ is $\delta\rify-\epsilon\rify$ equivariant,
and we have shown that R-ification is functorial.

Finally, for the assertion
that $\delta\rify\sify=\delta$,
just combine
\lemref{tilde alpha},
\lemref{c of r co},
and the fact that the isomorphism $\psi_A\:A^m\variso A$ is $\delta^m-\delta$ equivariant.
\end{proof}

Note that, as we mentioned earlier, the existence of an R-coaction whose standardization is $\delta$ does not imply that $\delta$ is maximal. However, for the present we are satisfied with the above R-ification functor on maximal coactions.

\begin{defn}
An R-coaction is \emph{maximal} if its standardization is maximal.
\end{defn}

Note that we are not just adapting the definition of maximality to R-coactions: we make no attempt to discuss the crossed product of an R-coaction (but see \cite{rae:representation}), so in particular it would not make sense to ask about injectivity of the canonical surjection
$\Phi\:A\rtimes_\delta G\rtimes_{\what\delta} G\to A\xt \KK$.

If $(B,\alpha)$ is an action,
then
we can directly produce the dual R-coaction
$\direct\alpha$ on the crossed product $B\rtimes_\alpha G$,
as in \defnref{dual R coaction}.
We are finally ready for our main result:
the direct approach is consistent with the R-ification of the dual \coaction\ $\what\alpha$,
as in \defnref{R of max}.

\begin{thm}\label{same}
Let $(B,\alpha)$ be an action, with dual \coaction\ $(B\rtimes_\alpha G,\what\alpha)$.
Then the two R-coactions
$\what\alpha\rify$ and
$\direct\alpha$
on $B\rtimes_\alpha G$ coincide.
\end{thm}

\begin{proof}
Let:
\begin{itemize}
\item
$(A,\delta)=(B\rtimes_\alpha G,\what\alpha)$

\item
$(E,\beta)=(A\rtimes_\delta G,\what\delta)$

\item
$\wilde A=E\rtimes_\beta G=A\rtimes_\delta G\rtimes_{\what\delta} G$

\item
$\pi=i_E\circ j_A\:A\to M(\wilde A)$

\item
$\phi=\pi\circ i_B\:B\to M(\wilde A)$

\item
$U=\pi\circ i_G^\alpha\:G\to M(\wilde A)$

\item
$\nu=i_E\circ j_G\:C_0(G)\to M(\wilde A)$

\item
$\iota=j_G\rtimes G\:\KK\to M(\wilde A)$

\item
$V=i_G^\beta\:G\to M(\wilde A)$

\item
$W=(\nu\xm\id)(w_G)\in M(\wilde A\xm C^*(G))$.
\end{itemize}

Note that
when working with $A\xm \KK\xm C^*(G))$,
we can
--- and in fact will need to --- sometimes replace either the first or the second $\xm$ by $\xt$
(and back again), but we cannot replace both by $\xt$ at the same time.
Of course, any such replacement must be accompanied by the insertion of parentheses around the minimal tensor product.

Since $\delta$ is maximal and $\KK$ is nuclear,
we can identify:
\begin{align*}
\wilde A&=A\xt \KK
\\
\pi&=(\id\xt\lambda)\circ\delta
\\
\nu&=1\xt M
\\
V&=1\xt\rho.
\end{align*}
Further, we can write
\begin{align*}
\phi(b)&=i_B(b)\xm 1\righttext{for}b\in B\\
U_s&=i_G^\alpha(s)\xm \lambda_s\righttext{for}s\in G,
\end{align*}
although again we can replace $\xm$ by $\xt$.
Also, we have observed previously that
\[
A^m=\relcom(\wilde A,\iota)=A\xm 1,
\]
and so
by \lemref{C tensor D},
\[
\relcom(\wilde A\xm C^*(G),\iota\xm 1)=A^m\xm C^*(G)=A\xm 1\xm C^*(G),
\]
and
\[
W=1\xm (M\xm \id)(w_G)\in M(A\xm \KK\xm C^*(G)).
\]
The R-coaction directly produced from $\what\beta$ is
given on the generators by:
\begin{align*}
\direct\beta\circ i_E&=i_E\xm 1\\
\direct\beta\circ i_G^\beta(s)&=i_G^\beta(s)\xm s.
\end{align*}
Thus:
\begin{align*}
\direct\beta\circ \pi&=\pi\xm 1\\
\direct\beta\circ \phi(b)&=i_B(b)\xm 1\xm 1\righttext{for}b\in B\\
\direct\beta(U_s)&=i_G^\alpha(s)\xm \lambda_s\xm 1\righttext{for}s\in G\\
\direct\beta\circ\nu&=\nu\xm 1\\
\direct\beta(V_s)&=1\xm \rho_s\xm s\righttext{for}s\in G.
\end{align*}
Now we perturb by the $\direct\beta$-cocycle $W$,
getting an exterior-equivalent R-coaction $\perturb\beta=\ad W\circ \direct\beta$.
For $b\in B$, $s\in G$, and $f\in C_0(G)$ we have:
\begin{align*}
\perturb\beta\circ\phi(b)&=\ad (1\xm M\xm \id)(w_G)\bigl(i_B(b)\xm 1\xm 1\bigr)
\\&=i_B(b)\xm 1\xm 1,
\end{align*}
\begin{align*}
\perturb\beta(U_s)&=\ad (1\xm M\xm \id)(w_G)\bigl(i_G^\alpha(s)\xm \lambda_s\xm 1\bigr)
\\&=i_G^\alpha(s)\xm \ad (M\xt \id)(w_G)\bigl(\lambda_s\xt 1\bigr)
\\&\hspace{.5in}\text{(swapping $\xm$ for $\xt$ in $M(\KK\xt C^*(G))$)}
\\&=i_G^\alpha(s)\xm (\lambda_s\xt s)
\\&\hspace{.5in}\text{(since $(\lambda,M)$ is a covariant representation of $\deltag$)}
\\&=i_G^\alpha(s)\xm \lambda_s\xm s
\\&\hspace{.5in}\text{(because $\lambda_s\xt s\in M(\KK\xt C^*(G))$)},
\end{align*}
\begin{align*}
\perturb\beta\circ \nu(f)&=\ad (1\xm M\xm \id)(w_G)\bigl(1\xm M_f\xm 1\bigr)
\\&=1\xm M_f\xm 1
\\&\hspace{.5in}\text{($(M\xm \id)(w_G)$ commutes with $C_0(G)\xm 1$)},
\end{align*}
and
\begin{align*}
\perturb\beta(V_s)&=\ad (1\xm M\xm \id)(w_G)\bigl(1\xm \rho_s\xm 1\bigr)
\\&=1\xm \ad (M\xt \id)(w_G)(\rho_s\xt s)
\\&\hspace{1in}\text{(swapping $\xm$ for $\xt$ again)}
\\&=1\xm (\rho_s\xt 1)
\\&\hspace{1in}\text{(because $\ad (M\xt \id)(w_G)\circ \what\rt=\triv$)}
\\&=1\xm \rho_s\xm 1
\righttext{(swapping $\xt$ back to $\xm$).}
\end{align*}

Now we compute the coaction $\epsilon=\relcom(\perturb\beta)$ on the relative commutant
$A^m=\relcom(\wilde A,\iota)$.
Since $A$ is generated by $i_B(B)$ and $i_G^\alpha(C^*(G))$,
we can compute on the generators:
for $b\in B$ and $s\in G$ we have
\begin{align*}
\epsilon(i_B(b)\xt 1)
&=\epsilon\circ\phi(b)
\\&=i_B(b)\xm 1\xm 1
\end{align*}
and
\begin{align*}
\epsilon(i_G^\alpha(s)\xt 1)
&=\epsilon((i_G^\alpha(s)\xt \lambda_s)(1\xt \lambda_{s\inv}))
\\&=\epsilon(i_G^\alpha(s)\xt \lambda_s)\epsilon(1\xt \lambda_{s\inv})
\\&=\epsilon(U_s)(1\xm \lambda_{s\inv}\xm 1)
\\&\hspace{1in}\text{(since $\epsilon$ is a \kcor)}
\\&=(i_g^\alpha(s)\xm \lambda_s\xm s)(1\xm \lambda_{s\inv}\xm 1)
\\&=i_G^\alpha(s)\xm 1\xm s,
\end{align*}
and therefore
\begin{align*}
(\psi\xm\id)\circ\epsilon(i_B(b)\xt 1)
&=(\psi\xm\id)(i_B(b)\xm 1\xm 1)
\\&=i_B(b)\xm 1
\\&=\delta\rify(i_B(b))
\\&=\delta\rify\circ\psi(i_B(b)\xt 1)
\end{align*}
and
\begin{align*}
&(\psi\xm\id)\circ\epsilon(i_G^\alpha(s)\xt 1)
\\&\quad=(\psi\xm\id)(i_G^\alpha(s)\xm 1\xm s)
\\&\quad=i_G^\alpha(s)\xm s
\\&\quad=\delta\rify(i_G^\alpha(s))
\\&\quad=\delta\rify\circ\psi(i_G^\alpha(s)\xt 1).
\end{align*}
Thus the diagram
\[
\xymatrix{
A^m \ar[r]^-\epsilon \ar[d]_\psi^\simeq
&M(A^m\xm C^*(G)) \ar[d]^{\psi\xm\id}_\simeq
\\
A \ar[r]_-{\delta\rify}
&M(A\xm C^*(G))
}
\]
commutes, and so
\[
\delta\rify=\what\alpha\rify=\direct\alpha.
\qedhere
\]
\end{proof}

We now have two possible interpretations of the notation $\deltagr$;
the following corollary assures us that
they are consistent.

\begin{cor}
The R-ification of the canonical \coaction\ $\deltag$ on $C^*(G)$ coincides with the R-coaction $\deltagr$ introduced in \defnref{R coaction}.
\end{cor}

\section{Conclusion}\label{conclusion}

In the proof of \thmref{same} we never seemed to need crossed products by R-coactions,
or even covariant representations of R-coactions.
But we did need a bit of the theory of cocycles for R-coactions,
in particular the fact that if $\delta$ is an R-coaction and $W$ is a $\delta$-cocycle, then $\ad W\circ\delta$ is an R-coaction.

Although it is possible --- as Raeburn does in \cite{rae:representation} --- to develop the elementary theory of covariant representations of an R-coaction $(A,\delta)$ on Hilbert space,
and further to define a crossed product $A\rtimes_\delta G$
(using representations on Hilbert space for the universal property),
it would be problematic to try to put this in the more general context of covariant representations in multiplier algebras.
One specific problem is that, if
$(A,\delta\sify)$ is the \coaction\ associated to an R-coaction $(A,\delta)$,
and if $(\pi,\mu)$ is a covariant representation of $\delta\sify$ in $M(B)$, we do not know whether there is a corresponding covariant representation of $\delta$. Consequently, we would not have a satisfactory theory of crossed products using representations in multiplier algebras.
Moreover, as discussed in \cite{rae:representation} and \cite{fulldual}, even if crossed products of R-coactions are defined (using Hilbert-space representations), they will be naturally isomorphic to the crossed products by the associated \coaction s.
Consequently, we made no attempt to discuss crossed-product duality for R-coactions.


\providecommand{\bysame}{\leavevmode\hbox to3em{\hrulefill}\thinspace}
\providecommand{\MR}{\relax\ifhmode\unskip\space\fi MR }
\providecommand{\MRhref}[2]{%
  \href{http://www.ams.org/mathscinet-getitem?mr=#1}{#2}
}
\providecommand{\href}[2]{#2}

\end{document}